\documentclass[11pt,reqno]{amsart}

\usepackage[T1]{fontenc}
\usepackage[a4paper,margin=29mm]{geometry}
\usepackage{lmodern}
\usepackage{microtype}
\microtypesetup{expansion=false}
\usepackage{amsmath,amssymb,amsthm}
\usepackage{mathtools}
\usepackage{mathrsfs}
\usepackage{enumitem}
\usepackage{tikz-cd}
\usepackage{xcolor}
\usepackage{aliascnt}
\usepackage{hyperref}
\usepackage[nameinlink,capitalize,noabbrev]{cleveref}

\hypersetup{
  colorlinks=true,
  linkcolor=blue!55!black,
  citecolor=green!40!black,
  urlcolor=blue!65!black
}

\numberwithin{equation}{section}

\newtheorem{theorem}{Theorem}[section]

\newaliascnt{proposition}{theorem}
\newtheorem{proposition}[proposition]{Proposition}
\aliascntresetthe{proposition}

\newaliascnt{lemma}{theorem}
\newtheorem{lemma}[lemma]{Lemma}
\aliascntresetthe{lemma}

\newaliascnt{corollary}{theorem}
\newtheorem{corollary}[corollary]{Corollary}
\aliascntresetthe{corollary}

\newaliascnt{conjecture}{theorem}

\aliascntresetthe{conjecture}

\newaliascnt{assumption}{theorem}

\aliascntresetthe{assumption}

\theoremstyle{definition}

\newaliascnt{definition}{theorem}
\newtheorem{definition}[definition]{Definition}
\aliascntresetthe{definition}

\newaliascnt{example}{theorem}

\aliascntresetthe{example}

\theoremstyle{remark}

\newaliascnt{remark}{theorem}
\newtheorem{remark}[remark]{Remark}
\aliascntresetthe{remark}

\crefname{proposition}{Proposition}{Propositions}
\Crefname{proposition}{Proposition}{Propositions}
\crefname{lemma}{Lemma}{Lemmas}
\Crefname{lemma}{Lemma}{Lemmas}
\crefname{corollary}{Corollary}{Corollaries}
\Crefname{corollary}{Corollary}{Corollaries}
\crefname{conjecture}{Conjecture}{Conjectures}
\Crefname{conjecture}{Conjecture}{Conjectures}
\crefname{assumption}{Assumption}{Assumptions}
\Crefname{assumption}{Assumption}{Assumptions}
\crefname{definition}{Definition}{Definitions}
\Crefname{definition}{Definition}{Definitions}
\crefname{example}{Example}{Examples}
\Crefname{example}{Example}{Examples}
\crefname{remark}{Remark}{Remarks}
\Crefname{remark}{Remark}{Remarks}

\setlist[enumerate]{leftmargin=2.4em}

\DeclareMathOperator{\Gr}{Gr}
\DeclareMathOperator{\Conv}{Conv}
\DeclareMathOperator{\Span}{span}

\newcommand{\C}{\mathbb C}
\newcommand{\Q}{\mathbb Q}
\newcommand{\Z}{\mathbb Z}
\newcommand{\R}{\mathbb R}
\newcommand{\A}{\mathbb A}
\newcommand{\cD}{\mathscr D}
\newcommand{\cE}{\mathcal E}
\newcommand{\cK}{\mathscr K}
\newcommand{\bH}{\mathbb H}
\newcommand{\bSigma}{\boldsymbol\Sigma}
\newcommand{\cbSigma}{\check{\boldsymbol\Sigma}}
\newcommand{\Boxop}{\operatorname{Box}}
\newcommand{\DRlog}{\operatorname{DR}_{\log}}
\newcommand{\Forb}{\mathcal F_{\mathrm{orb}}}
\newcommand{\Fst}{\mathcal F_{\mathrm{st}}}
 
\title[Stringy irregular Hodge numbers for toric LG data]
{Stringy Irregular Hodge Numbers for\\
Non-Simplicial Toric Landau--Ginzburg Data}
\author{Haoxu Wang}
\address{Morningside Center of Mathematics, Academy of Mathematics and
Systems Science, Chinese Academy of Sciences, Beijing 100190, China}
\email{krassotkinkolya@gmail.com}
\hypersetup{
  pdftitle={Stringy Irregular Hodge Numbers for Non-Simplicial
    Toric Landau--Ginzburg Data},
  pdfauthor={Haoxu Wang},
  pdfsubject={Stringy irregular Hodge numbers for non-simplicial
    toric Landau--Ginzburg data},
  pdfkeywords={toric face rings, Koszul--de Rham complexes,
    Landau--Ginzburg models, non-simplicial fans, crepant subdivisions,
    irregular Hodge numbers}
}
\date{}

\begin{document}
\raggedbottom

\begin{abstract}
We extend orbifold irregular Hodge theory from simplicial toric
Landau--Ginzburg models to data defined by possibly non-simplicial
stacky fans.  
For a Clarke mirror pair,
we construct a filtered Koszul--de Rham complex over toric face
rings.  
For simplicial ray-supported data,
a strict filtered comparison identifies this complex with the
Harder--Lee cellular residue complex, and hence it computes the
orbifold irregular Hodge numbers.

Using this complex, for a fixed possibly non-simplicial datum, 
we prove that the models obtained from 
its projective crepant simplicial subdivisions 
have the same orbifold irregular Hodge polynomial 
whenever the potential is nondegenerate at infinity.
This common value defines the stringy irregular Hodge numbers of the original datum.
\end{abstract}
 \maketitle
\setcounter{tocdepth}{1}
\tableofcontents

\section{Introduction}
\label{sec:introduction}

\subsection{Irregular Hodge theory and toric mirror symmetry}

Harder--Lee \cite{HarderLee} study mirror symmetry for toric
Landau--Ginzburg models on Deligne--Mumford stacks.  Their models are
pairs \((U,w)\) consisting of a smooth Deligne--Mumford stack \(U\) and a
regular function \(w:U\to\A^1\).  The twisted de Rham cohomology is
computed by \((\Omega_U^\bullet,d+dw\wedge)\).  This exponential twist
is generally irregular at infinity, and the relevant filtration is the
irregular Hodge filtration.  Yu introduced the filtration for smooth
quasi-projective varieties \cite{Yu}; its degeneration and its
mixed-Hodge-module interpretation are developed in
\cite{ESY,SabbahYu}; see also \cite{SabbahIHT}.

For a Deligne--Mumford Landau--Ginzburg model, applying the irregular
Hodge filtration to the connected components of the inertia stack and
incorporating the age shifts gives the orbifold irregular Hodge numbers
\(f_{\mathrm{orb}}^{\lambda,\mu}(U,w)\) and their generating series
\(\Forb(U,w;x,y)\).

The toric stack \(U\) is described by a stacky fan.  In the free-lattice
form used here, a stacky fan
\(\bSigma=(\Sigma,\{b_\rho\}_{\rho\in\Sigma(1)})\) consists of a
rational fan \(\Sigma\subset M_\R\) and a nonzero lattice vector
\(b_\rho\) on each ray.  When \(\Sigma\) is simplicial, this data gives a
smooth toric Deligne--Mumford stack \(U_{\bSigma}\) by the Cox
construction of Borisov--Chen--Smith \cite{BorisovChenSmith}.  The
nonprimitive ray vectors record the stack structure, and the associated
Box elements index the inertia sectors.

Let \(M\) and \(N=M^\vee\) be dual lattices, and let \(\bSigma\) and
\(\cbSigma\) be simplicial stacky fans in \(M\) and \(N\), respectively.
Following Clarke \cite{ClarkeDualFans}, the two fans form a dual pair
if
\begin{equation}\label{eq:intro-clarke-nonnegative-pairing}
 \langle |\Sigma|,|\check\Sigma|\rangle\geq0.
\end{equation}
Under this condition, the stacky rays of either fan define regular
Laurent monomials on the toric stack associated with the other.  Choose
generic ray-supported potentials \(g\) on \(U_{\bSigma}\) and
\(\check g\) on \(U_{\cbSigma}\) from the rays of \(\cbSigma\) and
\(\bSigma\), respectively.  The resulting Landau--Ginzburg models form
a stacky Clarke mirror pair.  Harder--Lee's mirror formula states that
\begin{equation}\label{eq:intro-harder-lee-mirror}
 f_{\mathrm{orb}}^{\lambda,\mu}(U_{\bSigma},g)
 =
 f_{\mathrm{orb}}^{d-\lambda,\mu}(U_{\cbSigma},\check g),
 \qquad
 d=\operatorname{rank}M,
\end{equation}
for all \(\lambda,\mu\in\Q\); see
\cite[Theorem~5.16]{HarderLee}.  Their proof computes both sides by a
cellular residue complex.  The question considered here is how this
construction behaves when the fans are not assumed to be simplicial.

\subsection{The stringy extension problem}

The classical theory of stringy Hodge numbers suggests what should
happen when the simpliciality assumption is removed.  Stringy Hodge
numbers extend Hodge-theoretic invariants to varieties with suitable
singularities; when a crepant resolution exists, they agree with the
Hodge numbers of the resolution and are therefore independent of its
choice \cite{BatyrevDais,BatyrevStringy}.  In the toric setting, inertia
sectors of a smooth toric Deligne--Mumford stack encode the corresponding
local stringy contributions \cite{BorisovChenSmith,Yasuda}.

Dropping simpliciality from a stacky fan leads to a related but
different geometric situation.  The natural toric quotient stack still
exists, but in general it is an Artin stack rather than a
Deligne--Mumford stack.  The orbifold irregular Hodge theory recalled
above does not directly assign invariants to this object.  We therefore
do not use the non-simplicial stack as a geometric model.  Instead, we
work with a combinatorial object \((\bSigma,f)\), where \(\bSigma\) is a
possibly non-simplicial stacky fan and \(f\) is a Laurent polynomial
that is regular on its toric variety.  We call this a toric
Landau--Ginzburg datum.  A projective crepant simplicial subdivision
\[
 \boldsymbol\Gamma\longrightarrow\bSigma
\]
produces a smooth toric Deligne--Mumford Landau--Ginzburg model
\((U_{\boldsymbol\Gamma},f)\), to which the Harder--Lee calculation and
the preceding Deligne--Mumford irregular Hodge theory apply.  The
stringy extension problem is to show that the resulting orbifold
irregular Hodge numbers are independent of the chosen subdivision.

\begin{theorem}[Crepant subdivision invariance]
\label{thm:intro-polytope-formula}
Let \((\bSigma,f)\) be a toric Landau--Ginzburg datum and put
\(Q:=P_{f,\infty}\).  Then
\[
 \Forb(U_{\boldsymbol\Gamma},h;x,y)
\]
is independent of the projective crepant simplicial subdivision
\(\boldsymbol\Gamma\to\bSigma\) and of the choice of \(h\) among
Newton nondegenerate potentials with Newton polytope at infinity
\(Q\).
\end{theorem}

We do not construct a motivic measure with values in irregular Hodge
structures or an irregular Hodge theory for toric Artin stacks.
The related invariance theorem of Qin--Zhang applies to a fixed smooth
quasi-projective variety and therefore does not compare the toric stacks
arising from different subdivisions \cite{QinZhang}.

\subsection{Outline of the proof}

To prove subdivision invariance, we associate to a Clarke mirror pair
a filtered Koszul--de Rham complex over the toric face rings of its
orthogonality fan.  More precisely, let \(\bSigma\) and
\(\cbSigma\) be stacky fans in dual lattices whose supports pair
nonnegatively, and put
\[
\Phi=(\Sigma\oplus\check\Sigma)_0
 :=\left\{
 c\oplus\check c\in\Sigma\oplus\check\Sigma:
 \langle c,\check c\rangle=0
 \right\}.
\]
For projective crepant subdivisions \(\Gamma \to \Sigma\) and
\(\check\Gamma \to \check\Sigma\), let
\[
\pi:\widehat\Phi:=(\Gamma\oplus\check\Gamma)_0\longrightarrow\Phi
\]
be the induced subdivision.  Given potentials \(g\) and \(\check g\)
on the two fans, we define on \(\Phi\) a filtered complex
\[
\left( \widehat{\mathscr C}^{\mathrm{dR},N;
 \Gamma,\check\Gamma}_\Phi(\check g,g), F^{\bullet} \right).
\]
Locally at \(\eta=c\oplus\check c\in\Phi\), the complex admits a
factorization as the tensor product of an effective Koszul complex
\(\cK_c^{\mathrm{eff},\Gamma}(\check g)\), determined by the monomials
of \(\check g\) lying in \(c\), and a twisted logarithmic de Rham complex
\(\cD_{\Phi,g}^{\check\Gamma}(\eta)\), determined by the monomials of
\(g\) lying in \(\check c\).  Symbols without subdivision superscripts,
such as
\[
 \widehat{\mathscr C}^{\mathrm{dR},N}_\Phi,
 \qquad
 \cK_c^{\mathrm{eff}}(\check g),
 \qquad
 \cD_{\Phi,g}(\eta)
\]
denote the corresponding objects for the trivial subdivisions.

First, when \(\Gamma\) and \(\check\Gamma\) are simplicial and the
potentials are supported on their rays, there is a strict filtered
quasi-isomorphism
between the Harder--Lee cellular residue complex and
\(\widehat{\mathscr C}^{\mathrm{dR},N}_{\widehat\Phi}\).  At
\(\widehat\eta=\gamma\oplus\check\gamma\in\widehat\Phi\), the factor
\(\cK_\gamma^{\mathrm{eff}}(\widehat{\check g})\) accounts for the Box
elements indexing the inertia sectors, while
\(\cD_{\widehat\Phi,\widehat g}(\widehat\eta)\) is the twisted
logarithmic de Rham complex on the corresponding toric stratum.  It
follows that the bigraded hypercohomology of the associated graded
complex computes the orbifold irregular Hodge numbers; see
\cref{thm:simplicial-comparison}.

Second, 
by \cref{thm:filtered-toric-face-direct-image}, there is a
canonical filtered
quasi-isomorphism
\begin{equation}\label{eq:intro-filtered-direct-image}
 \widehat{\mathscr C}^{\mathrm{dR},N;
 \Gamma,\check\Gamma}_\Phi(\check g,g)
 \simeq
 R\pi_*
 \widehat{\mathscr C}^{\mathrm{dR},N}_{\widehat\Phi}
 (\widehat{\check g},\widehat g).
\end{equation}
Thus the filtered complex on the simplicial refinement can be computed
on the original orthogonality fan.

Third, write
\[
 \mathscr C_\Phi^{\Gamma,\check\Gamma}(\check g,g)
 :=
\Gr_F\widehat{\mathscr C}^{\mathrm{dR},N;
 \Gamma,\check\Gamma}_\Phi(\check g,g)
\]
for the associated graded complex.  It is a toric face-ring version of
the double Koszul complex introduced by Borisov \cite{BorisovString}.
The complex
\(\mathscr C_\Phi^{\Gamma,\check\Gamma}(\check g,g)\) is a complex of
modules over the sheaf \(\mathcal A_\Phi\) of conewise polynomial
functions, and its differential is \(\mathcal A_\Phi\)-linear.  The
decomposition theorem for pure sheaves on fans
\cite{BBFK,BresslerLunts}, applied to its pure coefficient sheaf, gives
a primitive--support decomposition
\[
 \mathscr C_\Phi^{\Gamma,\check\Gamma}(\check g,g)
 \simeq
 \bigoplus_{\vartheta\in\Phi}
 V_\vartheta^{\Gamma,\check\Gamma}
 \otimes_\C\mathcal K_{\Phi,\vartheta}.
\]
Following \cite{BorisovMavlyutov}, the spaces
\(V_\vartheta^{\Gamma,\check\Gamma}\) are constructed from primitive
quotients.  They contain the dependence on the potentials and the
subdivisions, whereas the support complexes
\(\mathcal K_{\Phi,\vartheta}\) depend only on the orthogonality fan and
the lattice pairing.  Under the facewise regularity hypotheses, the
bigraded Hilbert series of
\(V_\vartheta^{\Gamma,\check\Gamma}\) is a product of local weighted
\(h^*\)-polynomials \cite{StapledonWeighted,KatzStapledonLocal} and is
therefore independent of the subdivisions and the potential
coefficients.  The primitive--support decomposition
then shows that the bigraded Hilbert series of
\[
 \bH^\bullet\!\left(
  \Phi,
  \mathscr C_\Phi^{\Gamma,\check\Gamma}(\check g,g)
 \right)
\]
has the same independence; see
\cref{prop:toric-face-pure-decomposition,prop:toric-face-hilbert}.

The three comparisons can be carried out simultaneously for a generic
choice of coefficients.  Coefficient invariance
\cite[Corollary~6.6]{WangCoefficient} then extends the result to every
Newton nondegenerate potential with the prescribed Newton polytope at
infinity.  Suppressing the potentials from the notation, the proof is
summarized by
\begin{align*}
 \Forb(U_{\boldsymbol\Gamma},h;x,y)
 &=
 \operatorname{Hilb}_{x,y}
 \bH^\bullet(\widehat\Phi,\mathscr C_{\widehat\Phi})\\
 &=
 \operatorname{Hilb}_{x,y}
 \bH^\bullet\!\left(
  \Phi,\mathscr C_\Phi^{\Gamma,\check\Gamma}
 \right)
 =
 \operatorname{Hilb}_{x,y}
 \bH^\bullet\!\left(
  \Phi,\mathscr C_\Phi
 \right).
\end{align*}
The first equality comes from the strict filtered comparison with the
Harder--Lee complex, the second from filtered direct image, and the
independence of the last expression follows from the primitive--support
decomposition.  Applying this argument to the spanning fans of
\(P_{\bSigma}\) and \(P_{f,\infty}\) proves
\cref{thm:intro-polytope-formula}.

\subsection{Organization of the paper}

\Cref{sec:lg-background} recalls orbifold irregular Hodge numbers and
toric Landau--Ginzburg data.  In
\cref{sec:koszul-de-rham-complexes} we construct the filtered
Koszul--de Rham complexes and prove the filtered direct-image theorem,
and \cref{sec:simplicial-realization} establishes the comparison with
the Harder--Lee residue complex.  The necessary results on pure sheaves
and weighted Ehrhart theory, followed by the primitive--support
decomposition and the intrinsic Hilbert-series calculation, are given
in
\cref{sec:pure-sheaves-weighted-ehrhart,sec:non-simplicial-associated-graded}.
Finally, \cref{sec:stringy-irregular-hodge-numbers} proves the polytope
formula and defines the stringy irregular Hodge numbers.
 \section{Toric Landau--Ginzburg models and orbifold irregular
Hodge numbers}
\label{sec:lg-background}

\subsection{Orbifold irregular Hodge numbers}
\label{subsec:lg-orbifold-conventions}

\begin{definition}[Orbifold irregular Hodge numbers]
\label{def:orbifold-irregular-hodge-numbers}
A \emph{smooth Deligne--Mumford Landau--Ginzburg model} is a pair
\((U,w)\), where \(U\) is a smooth separated Deligne--Mumford stack of
finite type over \(\C\) and \(w:U\to\A^1\) is a regular function.  Its
twisted de Rham cohomology is
\[
 H^k(U,w)
 :=
 \bH^k\!\left(
  U,(\Omega_U^\bullet,d+dw\wedge)
 \right).
\]
Write
\[
 IU:=U\times_{U\times U}U
\]
for the inertia stack of \(U\), where both maps to \(U\times U\) are
the diagonal, and let \(e:IU\to U\) be the evaluation morphism.  Its
geometric objects are pairs \((u,\gamma)\), where \(u\) is an object
of \(U\) and \(\gamma\in\operatorname{Aut}(u)\).  Write
\[
 IU=\coprod_{\mathfrak g\in\pi_0(IU)}U_{\mathfrak g}.
\]
For each connected component \(U_{\mathfrak g}\), let
\(a_{\mathfrak g}\) be its age and put
\(w_{\mathfrak g}:=(e^*w)|_{U_{\mathfrak g}}\).  For
\(\lambda,\mu\in\Q\), define
\begin{align}
 f_{\mathrm{orb}}^{\lambda,\mu}(U,w)
 &:={}
 \sum_{\mathfrak g\in\pi_0(IU)}
 \dim_\C
 \Gr_{F_{\mathrm{irr}}}^{\lambda-a_{\mathfrak g}}
 H^{\lambda+\mu-2a_{\mathfrak g}}
 (U_{\mathfrak g},w_{\mathfrak g}),
 \label{eq:orbifold-irregular-hodge-numbers}\\
 \Forb(U,w;x,y)
 &:={}
 \sum_{\lambda,\mu}
 f_{\mathrm{orb}}^{\lambda,\mu}(U,w)x^\lambda y^\mu.
 \label{eq:orbifold-irregular-hodge-polynomial}
\end{align}
Here a summand is understood to be zero unless
\(\lambda+\mu-2a_{\mathfrak g}\in\Z\).
\end{definition}

All generating series in this article are finite and are regarded as
elements of the integral group ring of \(\Q^2\):
\begin{equation}\label{eq:rational-bigraded-coefficient-ring}
 \Z[\Q^2]
 =\bigcup_{d\geq1}
  \Z[x^{\pm1/d},y^{\pm1/d}].
\end{equation}
The element \((\lambda,\mu)\in\Q^2\) is written as the monomial
\(x^\lambda y^\mu\).  For a finite-dimensional \(\Q^2\)-bigraded
vector space \(V\), and for a finite-dimensional \(\Q\)-graded vector
space \(W\), we write
\begin{equation*}\operatorname{Hilb}_{x,y}V
 :=\sum_{\lambda,\mu\in\Q}
   \dim_\C V^{\lambda,\mu}x^\lambda y^\mu,
 \qquad
 \operatorname{Hilb}W
 :=\sum_{\alpha\in\Q}\dim_\C W^\alpha T^\alpha.
\end{equation*}
Subscripts on \(\operatorname{Hilb}\) record the series variables; for
example, \(\operatorname{Hilb}_{U,V}\) denotes the analogous
bigraded series in \(U,V\).

Our notion of a smooth Deligne--Mumford Landau--Ginzburg model is the
one in \cite[Definition~2.2]{WangCompactification}.
The rational irregular Hodge filtration
\(F_{\mathrm{irr}}\) was introduced by Yu for smooth quasi-projective
varieties \cite[\S1]{Yu}.
We use the intrinsic irregular Hodge filtration of
\cite[Theorem~5.6]{WangCompactification},
\cite[Definition~3.9]{HarderLee}.

\subsection{Toric Landau--Ginzburg data and mirror pairs}
\label{subsec:toric-lg-data}

Let \(M\) and \(N=M^\vee\) be dual lattices.  For a lattice \(L\), let
\(T_L:=\operatorname{Spec}\C[L^\vee]\) be the torus with cocharacter
lattice \(L\).  For a rational cone \(c\subset L_\R\), write
\[
 c_\Q:=\Span_\Q(c\cap L),
 \qquad
 c_\C:=c_\Q\otimes_\Q\C.
\]

\begin{definition}[Stacky fan and radial support]
\label{def:stacky-fan}
A \emph{stacky fan} in \(M\) is a pair
\[
 \bSigma=(\Sigma,\{b_\rho\}_{\rho\in\Sigma(1)})
\]
consisting of a finite rational fan \(\Sigma\) of strongly convex
cones in \(M_\R\) and a nonzero lattice vector
\(b_\rho\in\rho\cap M\) on each ray.  
The stacky fan \(\bSigma\) is called
\emph{simplicial} if its underlying fan \(\Sigma\) is simplicial.

This notion is more general than the notion of a stacky fan in
\cite[Definition~3.1]{BorisovChenSmith}: we require neither that
\(\Sigma\) be simplicial nor that \(\bSigma[1]\) span \(M_\Q\).

We use the following notation, where \(c\in\Sigma\):
\[
 \bSigma[1]:=\{b_\rho:\rho\in\Sigma(1)\},
 \qquad
 \Delta_c
 :=\Conv\bigl(\{0\}\cup
       \{b_\rho:\rho\preceq c\}\bigr),
 \qquad
 P_{\bSigma}:=\bigcup_{c\in\Sigma}\Delta_c.
\]
The set \(P_{\bSigma}\) is called the \emph{radial stacky support}.

The stacky fan is \emph{stacky log-\(\Q\)-Gorenstein} if, for every
\(c\in\Sigma\), there is a rational linear form
\(\psi_c\in c_\Q^*\) such that
\(\psi_c(b_\rho)=1\) for every ray \(\rho\preceq c\).
These forms are unique and compatible on faces, and hence glue to a
positive rational conewise linear function
\begin{align*}
 \psi=\psi_{\bSigma}&:|\Sigma|\longrightarrow\R,\\
 \psi(m)&=\inf\{t\geq0:m\in tP_{\bSigma}\},
\end{align*}
called the \emph{height function}; it takes rational values on lattice
points.  
\end{definition}

\begin{definition}[Toric Landau--Ginzburg datum]
\label{def:toric-lg-datum}
A \emph{toric Landau--Ginzburg datum} is a pair \((\bSigma,f)\),
where \(\bSigma\) is a stacky fan in \(M\) and \(f\in\C[N]\).  
We impose the following conditions:
\begin{enumerate}[label=\textup{(\roman*)},leftmargin=*]
\item \(f\) is regular on the toric variety \(X_\Sigma\);
\item the underlying fan \(\Sigma\) is quasi-projective;
\item \(\bSigma\) is stacky log-\(\Q\)-Gorenstein;
\item the radial support \(P_{\bSigma}\) is convex;
\end{enumerate}
No full-dimensionality condition is imposed on either
\(P_{\bSigma}\) or \(P_{f,\infty}\).
The datum is called \emph{simplicial} if \(\bSigma\) is simplicial.

Write \(f=\sum_{n\in N}f_ny^n\).  
Its Newton polytope at infinity is
\begin{equation*}P_{f,\infty}
 :=\Conv\bigl(\{0\}\cup\operatorname{Supp}(f)\bigr)
 \subset N_\R, \qquad \text{where } \operatorname{Supp}(f) = \{n \in N | f_n \neq 0\}.
\end{equation*}
For any polytope \(R\) containing \(0\), a face
\(F\preceq R\) is called a \emph{face at infinity} if
\(0\notin F\).  For a face
\(F\preceq P_{f,\infty}\) at infinity, its face truncation and
logarithmic derivatives are
\[
 f_F:=\sum_{n\in F\cap N}f_ny^n,
 \qquad
 \partial_m f_F
 :=\sum_{n\in F\cap N}\langle m,n\rangle f_ny^n
 \quad(m\in M).
\]
The Laurent polynomial \(f\) is \emph{nondegenerate at
infinity} if, for every face \(F\) at infinity, the functions
\(\partial_m f_F\), for \(m\in M\), have no common zero on
the dense torus \(T_M=\operatorname{Spec}\C[N]\).

The toric Landau--Ginzburg datum \((\bSigma,f)\) is called 
nondegenerate at infinity if its potential \(f\) is nondegenerate at
infinity.
\end{definition}

\begin{remark}
The cells \(\Delta_c\) and all their faces form a coherent lattice
polytopal subdivision \(\mathcal S_{\bSigma}\) of \(P_{\bSigma}\):
it is obtained by subdividing the faces at infinity
and coning every resulting cell to \(0\).  Thus a toric
Landau--Ginzburg datum is equivalently specified by a lattice polytope
\(P\subset M_\R\) containing \(0\), such a coherent lattice polytopal
subdivision \(\mathcal S\) of \(P\), a lattice polytope
\(Q\subset N_\R\) containing \(0\) and satisfying
\[
 \langle P,Q\rangle\geq0,
\]
meaning that \(\langle m,n\rangle\geq0\) for every
\((m,n)\in P\times Q\),
and a Laurent polynomial \(f\in\C[N]\) with
\(P_{f,\infty}=Q\).  The fan \(\Sigma\) is recovered by taking the
cones over the boundary cells and using their vertices as the stacky
ray vectors; the displayed inequality is equivalent to the regularity
of \(f\) on \(X_\Sigma\).  The datum is simplicial precisely when
\(\mathcal S\) is a triangulation.
\end{remark}

For a simplicial datum \((\bSigma,f)\), put
\[
 M_\Sigma:=M\cap\Span_\R|\Sigma|.
\]
The stacky rays span \((M_\Sigma)_\Q\), so the Cox construction in
\(M_\Sigma\) gives a smooth toric Deligne--Mumford stack
\(\mathcal X_{\bSigma,M_\Sigma}\) \cite{BorisovChenSmith}.  Set
\begin{equation*}U_{\bSigma}
 :=\left[
  \bigl(T_M\times\mathcal X_{\bSigma,M_\Sigma}\bigr)
  \big/T_{M_\Sigma}
 \right],
\end{equation*}
where \(s\in T_{M_\Sigma}\) acts by
\(s\cdot(t,x)=(ts^{-1},s\cdot x)\).
Then \(U_{\bSigma}\) has coarse space \(X_\Sigma\). 
It is the stacky toric
Landau--Ginzburg model used in the Harder--Lee construction; in
particular,
\(f_{\mathrm{orb}}^{\lambda,\mu}(U_{\bSigma},f)\) is well defined.
If \(f\) is nondegenerate at infinity, orbifold coefficient
invariance \cite[Corollary~6.6]{WangCoefficient} shows that this number
depends only on \(U_{\bSigma}\) and \(P_{f,\infty}\).  
We denote the common value by
\[
 f_{\mathrm{orb}}^{\lambda,\mu}
 (U_{\bSigma},P_{f,\infty})
\]
and denote
\begin{equation*}\Forb(U_{\bSigma},P_{f,\infty};x,y)
 :=\sum_{\lambda,\mu}
 f_{\mathrm{orb}}^{\lambda,\mu}(U_{\bSigma},P_{f,\infty})x^\lambda y^\mu.
\end{equation*}

This allows us to use representatives adapted to the height grading.
For a lattice polytope \(Q\),
call \(g\in\C[N]\) a \emph{height-one representative of \(Q\)} if
its support lies on the faces at infinity and
\(P_{g,\infty}=Q\).

\begin{remark}
For a non-simplicial datum, we do not associate a geometric model
before choosing a projective crepant simplicial subdivision
\(\boldsymbol\Gamma\to\bSigma\), which produces the model
\((U_{\boldsymbol\Gamma},f)\) used
below.  A direct interpretation in terms of a toric Artin stack is
possible, but
will not be used here.  There is also a coarse geometric
interpretation: writing \(b_\rho=\beta_\rho v_\rho\), with \(v_\rho\)
primitive, the boundary
\[
 B_\Sigma
 :=\sum_{\rho\in\Sigma(1)}
   \left(1-\frac1{\beta_\rho}\right)D_\rho
\]
makes \((X_\Sigma,B_\Sigma)\) a toric klt pair, and the stacky
log-\(\Q\)-Gorenstein height records the \(\Q\)-Cartier divisor
\(K_{X_\Sigma}+B_\Sigma\).  Thus coarse log klt pairs provide a
parallel geometric background for the non-simplicial data.
\end{remark}

\begin{definition}[Mirror pair of toric Landau--Ginzburg data]
\label{def:toric-lg-mirror-pair}
Let
\[
 (\bSigma,f),
 \qquad
 (\cbSigma,\check f)
\]
be toric Landau--Ginzburg data in the dual lattices \(M\) and \(N\),
respectively.  They form a \emph{mirror pair} if
\begin{equation*}P_{f,\infty}=P_{\cbSigma},
 \qquad
 P_{\check f,\infty}=P_{\bSigma}.
\end{equation*}
\end{definition}

To construct mirror partners, we use the canonical fan associated
with a Newton polytope.
Let \(P\subset M_\R\) be a lattice polytope containing \(0\).
The \emph{spanning fan of \(P\) based at \(0\)} is
\begin{equation*}\Sigma_P
 :=\{\operatorname{Cone}(F)\mid
      F\text{ is a face at infinity of }P\}.
\end{equation*}
Here the empty face is included and
\(\operatorname{Cone}(\varnothing)=\{0\}\).
Its stacky vector on the ray through a nonzero vertex \(v\) of \(P\)
is \(v\) itself; the resulting stacky fan is denoted by
\(\boldsymbol\Sigma_P\).
It is clear that \(\boldsymbol\Sigma_P\) is quasi-projective and
stacky log-\(\Q\)-Gorenstein.

For \((\bSigma,f)\), 
choose a Laurent polynomial \(\check f\) with Newton polytope at
infinity \(P_{\bSigma}\). 
Then \(\check f\) is regular on \(X_{\Sigma_{P_{f,\infty}}}\) and
\((\boldsymbol\Sigma_{P_{f,\infty}},\check f)\) is a toric
Landau--Ginzburg datum and is a mirror partner of \((\bSigma,f)\).

For the remainder of the article, a fixed mirror pair will be written
as \((\bSigma,f)\) and \((\cbSigma,\check f)\), with the standing
polytope notation
\begin{equation}\label{eq:mirror-pair-polytope-notation}
 P:=P_{\bSigma}=P_{\check f,\infty},
 \qquad
 Q:=P_{\cbSigma}=P_{f,\infty}.
\end{equation}

A \emph{subdivision of stacky fans}
\(\boldsymbol\Gamma\to\bSigma\) consists of a subdivision
\(\Gamma\to\Sigma\) of the underlying fans, with the original stacky
vectors retained and with a nonzero lattice vector chosen on every
new ray.  It is \emph{crepant} if every new stacky vector lies on the
height-one locus of \(\psi_{\bSigma}\).
It is \emph{ray-preserving} if
\(\Gamma(1)=\Sigma(1)\), in which case it is automatically crepant.
The terms \emph{simplicial} and \emph{projective} refer to the
underlying fan and fan subdivision, respectively.

\begin{proposition}
\label{prop:simplicial-presentations-common-coarsening}
Let \(\bSigma\) be a quasi-projective stacky
log-\(\Q\)-Gorenstein fan with convex radial support \(P=P_{\bSigma}\).
\begin{enumerate}
\item
The natural map
\begin{equation*}\bSigma\longrightarrow\boldsymbol\Sigma_P
\end{equation*}
is a projective crepant subdivision.

\item
There is a projective ray-preserving crepant simplicial subdivision
\begin{equation*}
 \boldsymbol\Gamma\longrightarrow\bSigma.
\end{equation*}
\end{enumerate}
\end{proposition}

\begin{proof}
\begin{enumerate}
\item
It is clear that $\bSigma\longrightarrow\boldsymbol\Sigma_P$ is crepant.
A strictly convex support function witnessing the quasi-projectivity of \(\Sigma\) is
relatively strictly convex over \(\Sigma_P\), so the subdivision is
projective.

\item
The polytopes \(\Delta_c\) for every \(c\in\Sigma\) and all their faces form a finite
polytopal complex.  
Choose a total order on its vertices with \(0\) first and take the associated pulling
triangulation \cite[\S4.3.2]{Triangulations}.  It is coherent, restricts
compatibly to every face,
and triangulates each \(\Delta_c\) as a star at \(0\).  Deleting
\(0\) and taking cones therefore gives a simplicial fan with exactly
the original rays.  The coherent lifting function, normalized to
vanish at \(0\), extends homogeneously to a relatively strictly convex
support function; hence the subdivision is projective.  
\end{enumerate}
\end{proof}

A \emph{subdivision of a toric Landau--Ginzburg datum} is a morphism
\[
 (\boldsymbol\Gamma,\widehat f)
 \longrightarrow(\bSigma,f)
\]
induced by a subdivision of stacky fans such that
\((\boldsymbol\Gamma,\widehat f)\) is again a toric
Landau--Ginzburg datum, with
\(\widehat f\) equal to the pullback of \(f\) along the toric morphism
\(X_\Gamma\to X_\Sigma\).  Since this morphism is the identity on the
dense torus, \(\widehat f\) is the same Laurent polynomial as \(f\),
and in particular
\(P_{\widehat f,\infty}=P_{f,\infty}\).  
 \section{Koszul--de Rham complexes over toric face rings}
\label{sec:koszul-de-rham-complexes}

Retain the notation for mirror pairs from
\eqref{eq:mirror-pair-polytope-notation}, and choose height-one
representatives
\[
 g=\sum_{v\in N}\xi_vy^v\in\C[N],\qquad \check g=\sum_{u\in M}\check\xi_ux^u\in\C[M]
\]
of \(Q\) and \(P\), respectively.  The \emph{orthogonality fan} is
the subfan
\begin{equation*}\Phi=(\Sigma\oplus\check\Sigma)_0
 :=\left\{
 c\oplus\check c\in\Sigma\oplus\check\Sigma:
 \langle c,\check c\rangle=0
 \right\}.
\end{equation*}
We write the cones of \(\Phi\) as
\(\eta=(c,\check c)\), and denote the two projections by
\[
 p_M:\Phi\longrightarrow\Sigma,
 \qquad
 p_N:\Phi\longrightarrow\check\Sigma.
\]

Fix crepant subdivisions of Landau--Ginzburg data
\[
 (\boldsymbol\Gamma,\widehat g)
 \longrightarrow(\bSigma,g),
 \qquad
 (\check{\boldsymbol\Gamma},\widehat{\check g})
 \longrightarrow(\cbSigma,\check g).
\]
They induce a subdivision
\begin{equation*}\pi:\widehat\Phi
 :=(\Gamma\oplus\check\Gamma)_0
 \longrightarrow\Phi.
\end{equation*}
Indeed, for any $(\gamma,\check\gamma) \in \widehat\Phi$, if 
$\pi(\gamma,\check\gamma) = (c, \check c) \in \Sigma\oplus\check\Sigma$,
i.e. $c$ ($\check c$, resp.) is the smallest cone of $\Sigma$  ($\check \Sigma$, resp.)  containing $\gamma$ ($\check \gamma$, resp.),
then nonnegativity of the pairing makes
\(c\cap(\Span_\R\check\gamma)^\perp\) a face of \(c\) containing
\(\gamma\).
Hence $c = c\cap(\Span_\R\check\gamma)^\perp$, i.e. $c \perp \check \gamma$.
Therefore $\check c \cap c^{\perp}$ is a face of \(\check c\) containing
\(\check \gamma\).
Hence $\check c = \check c \cap c^{\perp}$, i.e. $(c, \check c) \in \Phi$.

\subsection{The filtered complex}

We use the standard description of sheaves on a fan as sheaves on a
finite space
\cite[\S\S0.A and 0.C]{BBFK}.  Regard a finite fan \(\Psi\) as
the Alexandrov space for the face order, so that its open subsets are
the subfans.  For \(\eta\in\Psi\), put
\[
 \langle\eta\rangle
 :=\{\gamma\in\Psi:\gamma\preceq\eta\},
 \qquad
 \partial\eta:=\langle\eta\rangle\setminus\{\eta\},
 \qquad
 \operatorname{Star}_\Psi(\eta)
 :=\{\gamma\in\Psi:\eta\preceq\gamma\}.
\]
A sheaf \(\mathcal F\) on \(\Psi\) is equivalently a collection of
stalks
\[
 \mathcal F_\eta
 :=\Gamma(\langle\eta\rangle,\mathcal F)
\]
and compatible face restrictions
\(\rho_{\eta\gamma}:\mathcal F_\eta\to\mathcal F_\gamma\) for
\(\gamma\preceq\eta\).  Thus, for a subfan \(\Psi_0\subseteq\Psi\),
\begin{equation*}\Gamma(\Psi_0,\mathcal F)
 =
 \left\{
  (s_\eta)_{\eta\in\Psi_0}:
  \rho_{\eta\gamma}(s_\eta)=s_\gamma
  \text{ whenever }\gamma\preceq\eta
 \right\}.
\end{equation*}

If \(\phi:\Psi'\to\Psi\) is order-preserving, then
\begin{equation*}(R^q\phi_*\mathcal F)_\eta
 =
 H^q\!\left(\phi^{-1}\langle\eta\rangle,\mathcal F\right).
\end{equation*}
For sheaves \(\mathcal F_i\) on \(\Psi_i\), their external product is
characterized by
\[
 (\mathcal F_1\boxtimes\mathcal F_2)_{(\eta_1,\eta_2)}
 =
 \mathcal F_{1,\eta_1}\otimes_\C\mathcal F_{2,\eta_2}.
\]

\begin{definition}\label{def:presented-toric-face-algebras}
A subdivision \(\pi_\Gamma:\Gamma\to\Sigma\) defines on \(\Sigma\)
the toric face algebra sheaf \(\mathscr S_\Sigma^\Gamma\) with stalk
\begin{equation*}(\mathscr S_\Sigma^\Gamma)_c
 =\C[c]^\Gamma
 =\bigoplus_{m\in c\cap M}\C x^m
\end{equation*}
and product
\begin{equation*}x^m\star_\Gamma x^{m'}
 =
 \begin{cases}
  x^{m+m'},&m,m'\text{ lie in a common cone of }\Gamma|_c,\\
  0,&\text{otherwise}.
 \end{cases}
\end{equation*}
For \(d\preceq c\), restriction fixes \(x^m\) when \(m\in d\) and
sends it to zero otherwise.  For the trivial subdivision, write
\[
 \mathscr S_\Sigma:=\mathscr S_\Sigma^{\mathrm{id}},
 \qquad (\mathscr S_\Sigma)_c=\C[c].
\]
The same notation, with checked symbols, is used on
\(\check\Sigma\).
\end{definition}

All these toric face algebras carry their height gradings:
\[
 \deg(x^m)=\psi(m),
 \qquad
 \deg(y^n)=\check\psi(n).
\]

\begin{lemma}\label{lem:semigroup-toric-face-direct-image}
There are canonical isomorphisms of algebra sheaves
\begin{equation*}(\pi_\Gamma)_*\mathscr S_\Gamma
 \cong\mathscr S_\Sigma^\Gamma,
 \qquad
 (\pi_{\check\Gamma})_*\mathscr S_{\check\Gamma}
 \cong\mathscr S_{\check\Sigma}^{\check\Gamma},
\end{equation*}
and both sheaves on the refined fans are acyclic for the corresponding
direct image functor.  Consequently,
\begin{equation*}R(\pi_\Gamma)_*\mathscr S_\Gamma
 \simeq\mathscr S_\Sigma^\Gamma,
 \qquad
 R(\pi_{\check\Gamma})_*\mathscr S_{\check\Gamma}
 \simeq\mathscr S_{\check\Sigma}^{\check\Gamma}.
\end{equation*}
\end{lemma}

\begin{proof}
Over \(\pi_{\Gamma}^{-1}\left(\langle c \rangle\right) = \Gamma|_c\), we have
$$
\Gamma\left(\Gamma|_c, \mathscr S_\Gamma\right)
 = 
\operatorname{span}_{\mathbb{C}}\left\{s_m : m\in c\cap M\right\},
$$
where the compatible monomial section \(s_m\) is equal
to \(x^m\) on the cones containing \(m\) and is zero elsewhere.  
The sections \(s_m\), for \(m\in c\cap M\), identify the algebra of
compatible sections with \(\C[c]^\Gamma\), including its product and
face maps.

For the higher direct images, decompose the restricted sheaf into its
monomial summands.  The \(m\)-summand is the constant rank-one sheaf on
\[
 \operatorname{Star}_{\Gamma|_c}(\gamma(m)),
\]
extended by zero, where \(\gamma(m)\) is the unique cone whose relative
interior contains \(m\).  This star has the minimum
\(\gamma(m)\), so its augmented order complex is contractible.  Every
monomial summand is therefore acyclic.  Since \(\Gamma|_c\) is finite,
the monomial decomposition commutes with cohomology.  This proves the
first derived isomorphism; the checked statement is identical.
\end{proof}

\begin{definition}\label{def:toric-face-koszul-de-rham-pair}
The following sheaves on \(\Phi\) form the factors of the
construction.

\begin{enumerate}[label=\textup{(\Roman*)},leftmargin=*,itemsep=.65\baselineskip]
\item
Put
$$
 \Lambda_N:=\bigwedge^\bullet N_\C
$$
and let \(\underline\Lambda_N\) be the corresponding constant sheaf.

\item The exterior algebra sheaf \(\cE_\Phi\) has stalk
\begin{equation*}\cE_{\Phi,\eta}:=\bigwedge^\bullet c^\perp \end{equation*}
with its standard grading and zero differential.  
For
\(\eta'=(d,\check d)\preceq\eta\), its face map is induced by
\(c^\perp\hookrightarrow d^\perp\).

\item 
For every homogeneous tensor
\(z=x^my^n\otimes\omega\) occurring below, where
$m \in M$, $n \in N$ and 
\(\omega\in\bigwedge^kN_\C\), put
\[
 a(z):=k+\psi(m)-\check\psi(n),
 \qquad
 r(z):=\psi(m)+\check\psi(n),
\]
with an absent monomial factor understood to have exponent zero, and
set
\begin{equation}\label{eq:n-polarized-total-degree}
 \deg_N(z):=k+2\psi(m)=a(z)+r(z).
\end{equation}

\item The full Koszul complex is
\begin{equation*}\cK_{\Phi,\check g}^{\Gamma}
 :=\left(
 p_M^{-1}\mathscr S_\Sigma^\Gamma
 \otimes\underline\Lambda_N,
 D_{\check g}
\right).
\end{equation*}
At \(\eta=(c,\check c)\),
\begin{equation*}\cK_{\Phi,\check g}^{\Gamma}(\eta)
 =\left(
 \C[c]^\Gamma\otimes \bigwedge^\bullet N_\C,
 D_{\check g}
\right),
\end{equation*}
where
\begin{align*}
 D_{\check g}(x^m\otimes\omega)
 &=
\sum_{u\in c\cap M}
 \check\xi_u
 (x^u\star_\Gamma x^m)\otimes\iota_u\omega.
\end{align*}
Exterior multiplication on the right makes this a right
dg-\(\cE_\Phi\)-module: every \(u\in c\) annihilates \(c^\perp\).

\item The residual twisted logarithmic de Rham complex is
\begin{equation*}\cD_{\Phi,g}^{\check\Gamma}
 :=\left(
 p_N^{-1}\mathscr S_{\check\Sigma}^{\check\Gamma}
 \otimes\cE_\Phi,
 d_{\log}+D_g \right).
\end{equation*}
Its stalk is
\begin{equation}\label{eq:residual-complex}
 \cD_{\Phi,g}^{\check\Gamma}(\eta)
 =\left(
 \C[\check c]^{\check\Gamma}\otimes\bigwedge^\bullet c^\perp, d_{\log}+D_g \right),
\end{equation}
where
\[
 d_{\log}(y^n\otimes\omega)=y^n\otimes(n\wedge\omega),
 \qquad
D_g(y^n\otimes\omega)
 =
 \sum_{v\in\check c\cap N}
 \xi_v
 (y^v\star_{\check\Gamma}y^n)\otimes(v\wedge\omega).
\]
Orthogonality gives \(\check c_\C\subseteq c^\perp\), 
so the differential is well-defined.  
Exterior multiplication on the left makes
\(\cD_{\Phi,g}^{\check\Gamma}\) a left dg-\(\cE_\Phi\)-module.

\item
The Koszul--de Rham complex is the tensor product
\begin{equation*}\widehat{\mathscr C}^{\mathrm{dR},N;
 \Gamma,\check\Gamma}_\Phi(\check g,g)
 :=
 \cK_{\Phi,\check g}^{\Gamma}
 \otimes_{\cE_\Phi}
 \cD_{\Phi,g}^{\check\Gamma}
 \cong
 p_M^{-1}\mathscr S_\Sigma^\Gamma
 \otimes p_N^{-1}\mathscr S_{\check\Sigma}^{\check\Gamma}
 \otimes\underline\Lambda_N,
\end{equation*}
Stalkwise, \(\Lambda_N\) is free over
\(\bigwedge c^{\perp}\), so the displayed
underived tensor product computes its derived counterpart.  
At \(\eta\),
\begin{align}\label{eq:total-local-differential}
 D(x^my^n\otimes\omega)
 ={}&(D_{\check g}+D_g+d_{\log})(x^my^n\otimes\omega) \notag\\
 ={}&\sum_{u\in c\cap M}\check\xi_u
 (x^u\star_\Gamma x^m)y^n\otimes\iota_u\omega
 \notag\\
 &+\sum_{v\in\check c\cap N}\xi_vx^m
 (y^v\star_{\check\Gamma}y^n)\otimes(v\wedge\omega)
 +x^my^n\otimes(n\wedge\omega).
\end{align}
The contraction term $D_{\check g}$ anticommutes with both wedge terms $D_g$ and $d_{\log}$ because
\(c\perp\check c\).  
Hence \(D^2=0\).

\item
Define the decreasing filtration stalkwise by
\[
 F^\lambda
 \widehat{\mathscr C}^{\mathrm{dR},N;
 \Gamma,\check\Gamma}_{\Phi,\eta}(\check g,g)
 :=
 \bigoplus_{a(z)\geq\lambda}\C z.
\]
With respect to the \((a,r)\)-degrees,
\[
 \deg(D_{\check g})=\deg(D_g)=(0,1),
 \qquad
 \deg(d_{\log})=(1,0).
\]

\end{enumerate}

For trivial subdivisions, omit \(\Gamma,\check\Gamma\) from all
superscripts; in particular, write
\(\widehat{\mathscr C}^{\mathrm{dR},N}_\Phi\) for the corresponding
semigroup algebra complex.
\end{definition}

\begin{remark}[Relation with Borisov's double Koszul complexes]
When \(\Sigma\) and \(\check\Sigma\) are the face fans of dual
reflexive Gorenstein cones and the subdivisions are trivial, the
complex just defined is a sheafwise contraction--wedge realization of
Borisov's double Koszul construction.  After passing to \(\Gr_F\), the
logarithmic term disappears, and global sections recover Borisov's
ordinary double Koszul complex
\cite[equation~(3.1)]{BorisovString}; this associated-graded
description is made explicit in
\cref{prop:graded-total-local-form}.  Before passing to \(\Gr_F\), the
logarithmic wedge term in \eqref{eq:total-local-differential}
corresponds to the extra term in Borisov's modified differential
\cite[Definition~4.6]{BorisovString}.
\end{remark}

Choose a positive integer \(L_0\) 
such that all grading indices below then lie in
\(\frac1{L_0}\Z\), and the corresponding Hilbert series lie in the ring
\(\Z[x^{\pm1/L_0},y^{\pm1/L_0}]\) of
\eqref{eq:rational-bigraded-coefficient-ring}.  We use the convention
\begin{equation}\label{eq:associated-graded-convention}
 F^{>\lambda}:=\sum_{\lambda'>\lambda}F^{\lambda'},
 \qquad
 \Gr_F^\lambda:=F^\lambda/F^{>\lambda},
 \qquad
 \Gr_F:=\bigoplus_{\lambda\in\frac1{L_0}\Z}\Gr_F^\lambda.
\end{equation}
The filtration is split according to the monomial exponents and the
exterior degree.  It is compatible with all face maps.

\begin{theorem}\label{thm:filtered-toric-face-direct-image}
For the subdivision
\(\pi:\widehat\Phi\to\Phi\), 
we have 
a canonical filtered morphism that is a
quasi-isomorphism on every filtered piece:
\begin{equation*}\widehat{\mathscr C}^{\mathrm{dR},N;
 \Gamma,\check\Gamma}_\Phi(\check g,g)
 \simeq
 R\pi_*
 \widehat{\mathscr C}^{\mathrm{dR},N}_{\widehat\Phi}
 (\widehat{\check g},\widehat g).
\end{equation*}
\end{theorem}

\begin{proof}
Fix \(\eta=(c,\check c)\in\Phi\).  The inverse image of its basic
face neighbourhood is
\[
 \Gamma|_c\times\check\Gamma|_{\check c},
\]
because every face of \(c\) is orthogonal to every face of
\(\check c\).  By the K\"unneth isomorphism for finite posets and
\cref{lem:semigroup-toric-face-direct-image}, its derived coefficient
sections are
\begin{align*}
 &R\Gamma(\Gamma|_c,\mathscr S_\Gamma)
 \otimes
 R\Gamma(\check\Gamma|_{\check c},\mathscr S_{\check\Gamma})
 \otimes\Lambda_N\\
 &\hspace{4em}\simeq
 \C[c]^\Gamma\otimes
 \C[\check c]^{\check\Gamma}\otimes\Lambda_N,
\end{align*}
concentrated in direct image degree zero.  These identifications commute
with restriction to faces.  Each fixed degree coset of the total
complex is bounded below, so these termwise acyclic sheaves compute its
derived direct image.

On compatible monomial sections, multiplication by \(x^u\) and \(y^v\)
becomes \(\star_\Gamma\) and \(\star_{\check\Gamma}\), respectively;
the contraction, wedge, and logarithmic operators on \(\Lambda_N\)
are unchanged.  Hence the induced differential on the direct image is
exactly
\eqref{eq:total-local-differential}.  
Since the height functions restrict to the refined cones, the
preceding identifications preserve the \((a,r)\)-bigrading and hence
the filtration \(F\).  The monomialwise K\"unneth argument applies on
every \(F^\lambda\), and compatibility with face restrictions gives
the claimed filtered quasi-isomorphism.
\end{proof}

\subsection{The effective Koszul complex and local decomposition}

For \(\eta=(c,\check c)\), put
\[
 \overline{N}_c:=N_\C/c^\perp
 \cong c_\C^\vee,
\]
and let \(\bar\iota_u\) be the contraction on
\(\bigwedge^\bullet\overline{N}_c\) induced by \(u\in c\).  The
\emph{effective Koszul complex} is
\begin{equation*}\cK_c^{\mathrm{eff},\Gamma}(\check g)
 :=\left(
 \C[c]^\Gamma\otimes\bigwedge^\bullet\overline{N}_c,
 \bar D_{\check g}
 \right),
\end{equation*}
where
\[
 \bar D_{\check g}(x^m\otimes\bar\omega)
 :=\sum_u\check\xi_u
 (x^u\star_\Gamma x^m)\otimes\bar\iota_u\bar\omega.
\]
For the trivial subdivision, write simply
\(\cK_c^{\mathrm{eff}}(\check g)\).

\begin{proposition}\label{prop:local-factorization}
Write
\[
 \widehat{\mathscr C}_\eta
 :=\widehat{\mathscr C}^{\mathrm{dR},N;
 \Gamma,\check\Gamma}_{\Phi,\eta}(\check g,g),
\]
and let \((c^\perp)^p\) denote the \(p\)-th power of the exterior ideal
generated by \(c^\perp\) in \(\Lambda_N\).  Then:

\begin{enumerate}[label=\textup{(\roman*)},leftmargin=*,itemsep=.5\baselineskip]
\item The finite decreasing filtration
\begin{equation*}G^p\cK_{\Phi,\check g}^{\Gamma}(\eta)
 :=\C[c]^\Gamma\otimes(c^\perp)^p
\end{equation*}
is stable under the differential, and
\begin{equation*}\Gr_G\cK_{\Phi,\check g}^{\Gamma}(\eta)
 \cong
 \cK_c^{\mathrm{eff},\Gamma}(\check g)
 \otimes\bigwedge^\bullet c^\perp.
\end{equation*}

\item It induces on the relative tensor product the filtration
\begin{equation*}G^p\widehat{\mathscr C}_\eta
 =\C[c]^\Gamma\otimes\C[\check c]^{\check\Gamma}
 \otimes(c^\perp)^p,
\end{equation*}
which is compatible with the total differential and the face maps.
Its associated graded is
\begin{equation*}\Gr_G\widehat{\mathscr C}_\eta
 \cong
 \cK_c^{\mathrm{eff},\Gamma}(\check g)
 \otimes
 \left(
 \C[\check c]^{\check\Gamma}
 \otimes\bigwedge^\bullet c^\perp,0
 \right).
\end{equation*}

\item Every splitting of
\[
 0\longrightarrow c^\perp\longrightarrow N_\C
 \longrightarrow\overline{N}_c\longrightarrow0
\]
splits the two filtrations and gives noncanonical decompositions of
filtered complexes
\begin{align}
 \cK_{\Phi,\check g}^{\Gamma}(\eta)
 &\cong
 \cK_c^{\mathrm{eff},\Gamma}(\check g)
 \otimes\bigwedge^\bullet c^\perp,
 \label{eq:full-effective-koszul-relation}\\
 \widehat{\mathscr C}_\eta
 &\cong
 \cK_c^{\mathrm{eff},\Gamma}(\check g)
 \otimes\cD_{\Phi,g}^{\check\Gamma}(\eta).
 \label{eq:local-split-factorization}
\end{align}
\end{enumerate}
The differential in both associated graded formulas is
\(\bar D_{\check g}\) on the effective Koszul factor and
zero on the remaining factors.  These statements do not require
simpliciality, projectivity, ray support, or nondegeneracy.
\end{proposition}

\begin{proof}
Put \(E=c^\perp\).  Since every \(u\in c\) annihilates \(E\),
contraction preserves the powers of the exterior ideal generated by
\(E\).  The canonical identification
\[
 \Gr_G^p\bigwedge^kN_\C
 \cong
 \bigwedge^{k-p}\overline{N}_c\otimes\bigwedge^pE
\]
intertwines the induced contraction with \(\bar\iota_u\), proving
\textup{(i)}.

After tensoring over \(\bigwedge E\), the induced filtration is again
the exterior ideal filtration.  It is compatible with faces because
\(c^\perp\subseteq d^\perp\) for \(d\preceq c\).  Since
\(\check c_\C\subseteq E\), both \(d_{\log}\) and
\(D_g\) raise the \(G\)-index and vanish
on \(\Gr_G\); the Koszul differential induces
\(\bar D_{\check g}\).  This proves \textup{(ii)}.

Finally, a splitting
\(N_\C\cong\overline{N}_c\oplus E\) separates the total differential
as
\[
 \bar D_{\check g}\otimes1
 +1\otimes(d_{\log}+D_g),
\]
which gives \textup{(iii)}.
\end{proof}
 \section{Simplicial realization of the Harder--Lee residue complex}
\label{sec:simplicial-realization}

Throughout this section, \(\bSigma\) and \(\cbSigma\) are simplicial
stacky fans, so no further subdivision is used in the toric face
algebras.  Write their stacky ray vectors as \(b_\rho\) and
\(\check b_{\check\rho}\), and choose ray-supported data
\begin{equation}\label{eq:hl-ray-supported-potential-data}
 g=\sum_{\check\rho\in\check\Sigma(1)}
 \xi_{\check\rho}y^{\check b_{\check\rho}},
 \qquad
 \check g
 :=\sum_{\rho\in\Sigma(1)}
 \check\xi_\rho x^{b_\rho},
\end{equation}
with all coefficients nonzero.  
For a simplicial cone \(c\) of dimension \(d\), with stacky ray
vectors \(b_1,\ldots,b_d\), put
\[
 \Boxop(c)
 :=
 \left\{
  p=\sum_{i=1}^dq_ib_i\in M:
  0\leq q_i<1
 \right\},
 \qquad
 \operatorname{age}(p):=\sum_{i=1}^dq_i.
\]
The inertia sectors of \(U_{\bSigma}\) are indexed by
\[
 \Boxop(\bSigma):=\bigcup_{c\in\Sigma}\Boxop(c),
\]
where each element is assigned to its minimal supporting cone.

\subsection{The Harder--Lee geometric model}

We recall the part of the Harder--Lee construction that gives a
geometric meaning to the complex defined in the preceding section.

Choose an integral strictly convex \(\check\Sigma\)-linear support
function \(\varphi\), whose existence follows from the
quasiprojectivity of \(\check\Sigma\).  Once a compatible refinement
of the normal fan is fixed, the construction of
\cite[Definition~9.7]{HarderLee} produces the fan \(\Sigma_\varphi\).
By \cite[Proposition~9.8]{HarderLee}, this fan satisfies
\cite[Assumption~9.1]{HarderLee}, and
\cite[Proposition~9.4]{HarderLee} then gives the toric quasi-stable
Landau--Ginzburg degeneration
\[
 \bigl(T(\Sigma_\varphi),D,w_\varphi,\pi\bigr),
 \qquad
 \pi:T(\Sigma_\varphi)\longrightarrow\A^1.
\]
Here \(D\) is the chosen horizontal toric boundary, containing the
reduced pole divisor of \(w_\varphi\), and \(w_\varphi\) is the
\(t\)-weighted degenerating potential.  
For \(0<|t|\ll1\), put
\[
 D_t:=D\cap\pi^{-1}(t),\qquad
 (U_t,w_t):=
 \bigl(\pi^{-1}(t)\setminus D_t,
       w_\varphi|_{\pi^{-1}(t)\setminus D_t}\bigr).
\]
We identify this reference nearby fibre with
\((U_{\bSigma},g)\).
The central
fibre \(X_0:=\pi^{-1}(0)\) is a reduced orbifold normal-crossings divisor whose
strata are toric.

Let \(\Lambda_{\mathrm{orb}}^\bullet\) be the twisted 
logarithmic de Rham complex on the central fibre.
More precisely, 
by \cite[Remark~9.23]{HarderLee}, the relevant components of the
relative inertia stack are quasi-stable degenerations of these
sectors.  If \(\Lambda_p^\bullet\) denotes the central restriction of
the relative twisted logarithmic de Rham complex on the \(p\)-sector,
viewed on \(X_0\) through evaluation, then
\[
 \Lambda_{\mathrm{orb}}^\bullet
 :=
 \bigoplus_{p\in\Boxop(\bSigma)}
 e_p\,\Lambda_p^\bullet.
\]
Here \(e_p\) is the formal label of the sector indexed by \(p\).

The degeneration results of
\cite[Theorem~8.12, Proposition~8.17, and
Theorem~9.21\textup{(1)}]{HarderLee}, applied componentwise as in
\cite[Remark~9.23]{HarderLee}, identify the orbifold irregular Hodge
numbers of the nearby fibre $U_t$ with the filtered cohomology of
the central complex:
\begin{align}\label{eq:hl-irregular-hodge-central-fibre}
 f_{\mathrm{orb}}^{\lambda,n-\lambda}(U_{\bSigma},g)
 =
 \dim_\C\Gr^\lambda_{F_{\mathrm{irr}}}
 \bH^n(X_0,\Lambda_{\mathrm{orb}}^\bullet).
\end{align}

The relevant intersections of the central and horizontal toric strata
are indexed by the face poset
\[
 \widetilde\Phi:=S_{T,D,0}
\]
of the refined tropical cell complex
\[
 T_{\Sigma_\varphi}(\Sigma_{0,D},w_\varphi)_0
\]
constructed in \cite[Section~9.3]{HarderLee}.
A cell of \(\widetilde\Phi\) is written
\((c,\sigma)\) and corresponds to a toric stratum \(T(\sigma)\) of the
central fibre.
Under the identification of \cite[Proposition~5.13]{HarderLee}, the
refinement gives the forgetful order-preserving map
\[
 s:\widetilde\Phi\longrightarrow\Phi, \qquad (c,\sigma) \mapsto (c,\check c).
\]

Restricting and taking residues of the relative
twisted logarithmic de Rham complex gives a filtered complex
\(\Lambda_{(c,\sigma)}^\bullet\) on this
stratum.  
Combining the toric description of
\cite[equations~(26)--(31)]{HarderLee}, the semigroup presentation
in the proof of \cite[Proposition~3.12]{HarderLee}, and the
Clarke-pair identification of
\cite[Proposition~5.13]{HarderLee}, we obtain
a strict filtered isomorphism
\[
 \Gamma\!\left(
  T(\sigma),\Lambda_{(c,\sigma)}^\bullet
 \right)
\cong
 \left(
  \C[\check c]\otimes\bigwedge^\bullet c^\perp,
  d_{\log}+D_g
 \right)
 =:
 \DRlog(c,\check c;g).
\]
Let \(\mathscr G_{\mathrm{orb}}^\bullet\) denote the cellular diagram
on \(\widetilde\Phi\) whose value at a cell
\((c,\sigma)\) is 
\[
 \bigl(\mathscr G_{\mathrm{orb}}^\bullet\bigr)_{(c,\sigma)}
 :=
 \bigoplus_{p\in\Boxop(c)}
 e_p\,
 \Gamma\!\left(
  T(\sigma),\Lambda_{(c,\sigma)}^\bullet
 \right)
 \cong
 \bigoplus_{p\in\Boxop(c)}
 e_p\,\DRlog(c,\check c;g).
\]
Its incidence morphisms are the maps
\(\alpha((c,\sigma),(c',\sigma'))\) of
\cite[equations~(27)--(28)]{HarderLee}.
Extend the grading convention of
Definition~\ref{def:toric-face-koszul-de-rham-pair} to the orbifold
sector labels by assigning \(e_p\) the \(M\)-height
\(\operatorname{age}(p)\).  Thus, for
\(\omega\in\bigwedge^kc^\perp\),
\begin{equation}\label{eq:orbifold-residue-degree-and-filtration}
 a(e_py^n\otimes\omega)
 =
 k+\operatorname{age}(p)-\check\psi(n),
 \qquad
 \deg_N(e_py^n\otimes\omega)
 =
 k+2\operatorname{age}(p).
\end{equation}
Hence the sector label \(e_p\) contributes
\(\operatorname{age}(p)\) to the filtration index and
\(2\operatorname{age}(p)\) to the internal cohomological degree.
These are respectively the shift of the irregular Hodge filtration 
and the Chen--Ruan cohomological degree shift;
compare \cite[Sections~2.4 and 4.4 and
equations~(24), (31)]{HarderLee}.

The residue resolution of
\cite[Proposition~9.19]{HarderLee}, together with the acyclicity
of its toric terms from \cite[Theorem~3.11]{HarderLee}, gives,
componentwise as in \cite[Remark~9.23]{HarderLee}, 
a filtered quasi-isomorphism
\begin{align}\label{eq:hl-residue-route}
 R\Gamma(X_0,\Lambda_{\mathrm{orb}}^\bullet)
 \simeq
 \operatorname{Tot}C^\bullet_{\mathrm{cell}}
 \bigl(\widetilde\Phi,\mathscr G_{\mathrm{orb}}^\bullet\bigr).
\end{align}
Here,
with the cellular convention of
\cite[Definition~4.10 and Proposition~9.12]{HarderLee}, 
\[
 C_{\mathrm{cell}}^q
 \bigl(\widetilde\Phi,\mathscr G_{\mathrm{orb}}^\bullet\bigr)
 =
 \bigoplus_{\substack{
  (c,\sigma)\in\widetilde\Phi^{\mathrm{cpt}}\\
  \dim\sigma-\dim c=q}}
 \bigl(\mathscr G_{\mathrm{orb}}^\bullet\bigr)_{(c,\sigma)}.
\]
Here 
\[
 \widetilde\Phi^{\mathrm{cpt}}
 :=
 S^{\mathrm{cpt}}_{T,D}
 \subseteq \widetilde\Phi
\]
is the subposet formed by compact cells.

The displayed local complexes and incidence maps depend only on
\(\eta=s(c,\sigma)=(c,\check c)\), and hence descend to the filtered
cellular diagram \(\mathscr R_{\Phi,\mathrm{orb}}^\bullet\) on
\(\Phi\) given by
\[
 \bigl(\mathscr R_{\Phi,\mathrm{orb}}^\bullet\bigr)_\eta
 :=
 \bigoplus_{p\in\Boxop(c)}
 e_p\,\DRlog(c,\check c;g).
\]
Consequently, the
preceding local identifications assemble to a strict filtered
isomorphism
\begin{equation}\label{eq:geometric-residue-pullback}
 \mathscr G_{\mathrm{orb}}^\bullet
 \cong
 s^*\mathscr R_{\Phi,\mathrm{orb}}^\bullet,
\end{equation}
the chain-level orbifold counterpart of the pullback relation used in
the proof of \cite[Proposition~9.15]{HarderLee}.

\subsection{The Box--Koszul comparison}

\begin{proposition}\label{prop:koszul-box}
For \(\eta=(c,\check c)\), write
\[
 \check g_c
 =\sum_{i=1}^d\check\xi_ix^{b_i},
 \qquad \check\xi_i\ne0.
\]
Then the two local factors are computed separately by
\begin{align}
 H_i\!\left(
 \cK_c^{\mathrm{eff}}(\check g)
 \right)
 &\cong
 \begin{cases}
  \mathbf B_c^{\mathrm{alg}},&i=0,\\
  0,&i\ne0,
 \end{cases}
 \label{eq:box-effective-koszul-homology}\\
 \mathbf B_c^{\mathrm{alg}}
 &:=
 \C[c]/(x^{b_1},\ldots,x^{b_d})
 \cong
 \bigoplus_{p\in\Boxop(c)}\C x^p,
 \label{eq:box-algebraic-quotient}\\
 \cD_{\Phi,g}(\eta)
 &=\DRlog(c,\check c;g).
 \label{eq:box-residual-factor}
\end{align}
Here \(H_i\) denotes Koszul homology; 
Since \(\psi(p)=\operatorname{age}(p)\), the summand
\(\C x^p\) has the same filtration and internal cohomological shifts
as the sector label \(e_p\) in
\eqref{eq:orbifold-residue-degree-and-filtration}.

Consequently, the maps
\[
 e_p\omega\longmapsto x^p\otimes\omega
\]
assemble to a strict filtered quasi-isomorphism
\begin{equation}\label{eq:box-residue-koszul-comparison}
 \mathscr R_{\Phi,\mathrm{orb}}^\bullet
 \xrightarrow{\ \sim\ }
 \widehat{\mathscr C}^{\mathrm{dR},N}_\Phi
 (\check g,g).
\end{equation}
Thus the Koszul factor is replaced by its homological degree-zero
Koszul quotient,
while the full residual twisted logarithmic de Rham complex, including
its coefficients and differential, is retained.
\end{proposition}

\begin{proof}
Since \(c\) is simplicial, every \(m\in c\cap M\) has a unique
expression
\[
 m=p+\sum_i n_ib_i,
 \qquad
 p\in\Boxop(c),\quad n_i\in\Z_{\geq0}.
\]
Hence, for
\(S_c=\C[x^{b_1},\ldots,x^{b_d}]\),
\[
 \C[c]=
 \bigoplus_{p\in\Boxop(c)}x^pS_c.
\]
Thus \(\C[c]\) is free over the polynomial algebra \(S_c\), and
\(\check\xi_1x^{b_1},\ldots,\check\xi_dx^{b_d}\) is a regular
sequence.  This proves
\eqref{eq:box-effective-koszul-homology} and
\eqref{eq:box-algebraic-quotient}.

The standard homogeneous contraction of the polynomial Koszul
complex respects the filtration.  Since
\(\psi(p)=\operatorname{age}(p)\), 
the map
\(e_p\omega\mapsto x^p\otimes\omega\)
preserves both \(a\) and \(\deg_N\).

By \eqref{eq:residual-complex} and
\eqref{eq:local-split-factorization}, this map is locally the Box
quasi-isomorphism tensored with the identity on
\(\DRlog(c,\check c;g)\).  It is canonical because every
\(b_i\in c\) annihilates \(c^\perp\), and it commutes with the face
maps, which retain precisely the same Box sectors.  The local maps
therefore assemble to the strict filtered quasi-isomorphism
\eqref{eq:box-residue-koszul-comparison}.
\end{proof}

Give the associated graded hypercohomology its irregular Hodge
bigrading by
\begin{equation}\label{eq:simplicial-associated-graded-bigrading}
 \bH^{\lambda,\mu}\!\left(
 \Phi,\Gr_F \widehat{\mathscr C}_{\Phi,\mathrm{orb}}
 \right)
 :=
 \bH^{\lambda+\mu}\!\left(
 \Phi,\Gr_F^\lambda\widehat{\mathscr C}_{\Phi,\mathrm{orb}}
 \right).
\end{equation}
If \(s\) denotes the cellular cochain degree, then a homogeneous
class of local \((a,r)\)-degree has irregular Hodge bidegree
\begin{equation}\label{eq:irregular-hodge-grading-convention}
 (\lambda,\mu)=(a,r+s),
 \qquad
 \lambda+\mu=a + r+s = \deg_N+s.
\end{equation}

\begin{theorem}\label{thm:simplicial-comparison}
For the simplicial ray-supported data
\eqref{eq:hl-ray-supported-potential-data}, put
\[
 \widehat{\mathscr C}_{\Phi,\mathrm{orb}}
 :=
 \widehat{\mathscr C}^{\mathrm{dR},N}_\Phi
 (\check g,g),
 \qquad
 \mathscr C_\Phi(\check g,g)
 :=
 \Gr_F\widehat{\mathscr C}_{\Phi,\mathrm{orb}}.
\]
There are strict filtered quasi-isomorphisms
\begin{equation}\label{eq:simplicial-filtered-comparison}
 Rs_*\mathscr G_{\mathrm{orb}}^\bullet
 \simeq
 \widehat{\mathscr C}_{\Phi,\mathrm{orb}},
 \qquad
 R\Gamma(X_0,\Lambda_{\mathrm{orb}}^\bullet)
 \simeq
 R\Gamma(\Phi,\widehat{\mathscr C}_{\Phi,\mathrm{orb}}),
\end{equation}
and
\begin{equation}\label{eq:simplicial-irregular-hodge-polynomial}
 \Forb(U_{\bSigma},g;x,y)
 =
 \operatorname{Hilb}_{x,y}
 \bH^{\bullet,\bullet}\!\left(
 \Phi,\mathscr C_\Phi(\check g,g)
 \right).
\end{equation}
The chain-level comparison retains all coefficients of \(g\).  The
resulting numerical identity extends from this ray-supported potential
to all potentials that are nondegenerate at infinity and have
the same Newton polytope by coefficient invariance.
\end{theorem}

\begin{proof}
Combining \eqref{eq:geometric-residue-pullback} with
\eqref{eq:box-residue-koszul-comparison} gives
\[
 \mathscr G_{\mathrm{orb}}^\bullet
 \simeq
 s^*\widehat{\mathscr C}_{\Phi,\mathrm{orb}}.
\]
Over every basic face neighbourhood, \(s\) is a cellular
subdivision.  Applying the subdivision argument from the proof of
\cite[Proposition~9.15]{HarderLee} to each filtered piece therefore
gives
\[
 \widehat{\mathscr C}_{\Phi,\mathrm{orb}}
 \simeq
 Rs_*s^*\widehat{\mathscr C}_{\Phi,\mathrm{orb}},
\]
and hence the first comparison in
\eqref{eq:simplicial-filtered-comparison}.  Together with the residue
resolution \eqref{eq:hl-residue-route}, this also gives
\[
 R\Gamma(X_0,\Lambda_{\mathrm{orb}}^\bullet)
 \simeq
 R\Gamma(\widetilde\Phi,\mathscr G_{\mathrm{orb}}^\bullet)
 \simeq
 R\Gamma(\Phi,\widehat{\mathscr C}_{\Phi,\mathrm{orb}}).
\]
The degree computation in Proposition~\ref{prop:koszul-box} shows
that these comparisons respect \(a\), \(\deg_N\), 
and hence the irregular Hodge bigrading \eqref{eq:simplicial-associated-graded-bigrading}.

It remains only to commute the associated graded with
hypercohomology.  By
\cite[Theorem~9.21\textup{(1)}]{HarderLee},
\[
 \bH^n(X_0,F_{\mathrm{irr}}^\lambda\Lambda^\bullet)
 \longrightarrow
 \bH^n(X_0,\Lambda^\bullet)
\]
is injective.  The same holds componentwise on the relative inertia
stack by \cite[Remark~9.23]{HarderLee}, and hence for
\(\Lambda_{\mathrm{orb}}^\bullet\).  Transporting these injections
through \eqref{eq:simplicial-filtered-comparison}, the standard exact
sequence for two consecutive filtered pieces gives
\[
 \Gr_F^\lambda
 \bH^n(\Phi,\widehat{\mathscr C}_{\Phi,\mathrm{orb}})
 \cong
 \bH^n\!\left(
  \Phi,\Gr_F^\lambda
  \widehat{\mathscr C}_{\Phi,\mathrm{orb}}
 \right).
\]
Combining this with
\eqref{eq:hl-irregular-hodge-central-fibre} and summing the
dimensions with weights \(x^\lambda y^\mu\) proves
\eqref{eq:simplicial-irregular-hodge-polynomial} for the
ray-supported potential.  Coefficient invariance
\cite[Corollary~6.6]{WangCoefficient} gives the stated extension.
\end{proof}

 \section{Pure sheaves on fans and weighted Ehrhart theory}
\label{sec:pure-sheaves-weighted-ehrhart}

\subsection{Sheaves on fans and conewise polynomials}
\label{subsec:sheaves-on-fans}

The sheaf of conewise polynomial rings is
\begin{equation*}\mathcal A_{\Psi,\eta}
 :=
 \operatorname{Sym}(\eta_\C^*),
\end{equation*}
with the evident face restrictions.  We place linear forms in degree
one; this is half the cohomological grading used in \cite{BBFK}.
For the orthogonality fan,
\begin{equation*}\mathcal A_{\Phi,(c,\check c)}
 =
 \mathcal A_{\Sigma,c}\otimes
 \mathcal A_{\check\Sigma,\check c},
 \qquad
 \deg c_\C^*=(1,0),\quad
 \deg\check c_\C^*=(0,1).
\end{equation*}
If \(G\) is a finitely generated graded
\(\mathcal A_{\Psi,\eta}\)-module, write
\begin{equation*}\overline G:=G/\mathfrak m_\eta G,
\end{equation*}
where \(\mathfrak m_\eta\) is the ideal of positive-degree
polynomials.  When \(G\) is free, \(\overline G\) is its space of
minimal homogeneous generators.

\begin{definition}[Pure sheaves]
\label{def:pure-sheaves}
A graded \(\mathcal A_\Psi\)-module sheaf \(\mathcal F\) is
\emph{pure} if every stalk is finitely generated and free and every
boundary map
\[
 \mathcal F_\eta\longrightarrow
 \Gamma(\partial\eta,\mathcal F)
\]
is surjective.  The latter condition is the form of flabbiness for
finite fans \cite[Definition~2.1]{BBFK}.
\end{definition}

Pure sheaves provide a combinatorial analogue of the geometric
decomposition theorem for pure perverse sheaves.
Normalized simple pure sheaves play the
role of intersection-cohomology complexes.  Every pure sheaf splits
into these simple pieces with graded multiplicities; the splitting is
not canonical, but the multiplicity spaces are.

\begin{theorem}\label{thm:pure-sheaf-decomposition}
\begin{enumerate}[label=\textup{(\roman*)}]
\item
For every \(\vartheta\in\Psi\), there is, up to isomorphism, a unique
normalized simple pure sheaf \(\mathcal L_\Psi^\vartheta\),
characterized by \cite[Remark~2.2b]{BBFK}
\[
 \operatorname{supp}\mathcal L_\Psi^\vartheta
 =\operatorname{Star}_\Psi(\vartheta),
 \qquad
 (\mathcal L_\Psi^\vartheta)_\vartheta
 =\mathcal A_{\Psi,\vartheta},
\]
and by requiring, for every \(\eta\succ\vartheta\) in its support,
that the boundary map
\[
 (\mathcal L_\Psi^\vartheta)_\eta
 \longrightarrow
 \Gamma(\partial\eta,\mathcal L_\Psi^\vartheta)
\]
be a minimal free cover.  
If \(\Psi\) is simplicial, \(\mathcal L_\Psi^\vartheta\) is
the conewise polynomial sheaf on
\(\operatorname{Star}_\Psi(\vartheta)\), extended by zero
\cite[Proposition~1.4 and Remark~2.2b(ii)]{BBFK}.

\item
Every pure sheaf \(\mathcal F\) admits a noncanonical decomposition
\cite[Theorem~2.3]{BBFK}:
\begin{equation}\label{eq:pure-sheaf-decomposition}
 \mathcal F
 \simeq
 \bigoplus_{\vartheta\in\Psi}
 \mathcal W_\vartheta(\mathcal F)
 \otimes_\C\mathcal L_\Psi^\vartheta.
\end{equation}
The graded multiplicity spaces are canonically determined by
\begin{equation}\label{eq:pure-multiplicity-kernel}
 \mathcal W_\vartheta(\mathcal F)
 =
 \ker\!\left(
  \overline{\mathcal F_\vartheta}
  \longrightarrow
  \overline{\Gamma(\partial\vartheta,\mathcal F)}
 \right).
\end{equation}
The grading shifts from \textup{(i)} are recorded by these graded
multiplicity spaces.
\end{enumerate}
\end{theorem}

\begin{lemma}\label{lem:orthogonal-external-product-simple}
Let \(i:\Phi\hookrightarrow\Sigma\times\check\Sigma\) be the
inclusion.  For \(\sigma\in\Sigma\) and
\(\check\sigma\in\check\Sigma\),
\begin{equation}\label{eq:orthogonal-external-product-simple}
 i^*\bigl(
  \mathcal L_\Sigma^\sigma\boxtimes
  \mathcal L_{\check\Sigma}^{\check\sigma}
 \bigr)
 \cong
 \begin{cases}
  \mathcal L_\Phi^{(\sigma,\check\sigma)},
  &\sigma\perp\check\sigma,\\
  0,&\sigma\not\perp\check\sigma.
 \end{cases}
\end{equation}
\end{lemma}

\begin{proof}
If \(\sigma\not\perp\check\sigma\), no cone of \(\Phi\) belongs to
the support of the restricted external product.  Otherwise, put
\(\vartheta=(\sigma,\check\sigma)\).  For every
\(\eta=(c,\check c)\succeq\vartheta\),
\[
 [\vartheta,\eta]
 \cong[\sigma,c]\times[\check\sigma,\check c].
\]
The boundary of this product interval is the union of its two
factorwise boundaries.  Consequently, the recursive minimal free
cover defining a normalized simple extension commutes with the
external product.  The restricted sheaf is therefore supported on
\(\operatorname{Star}_\Phi(\vartheta)\), has initial stalk
\(\mathcal A_{\Phi,\vartheta}\), and satisfies the defining
minimality condition.  Uniqueness gives the result.
\end{proof}

\subsection{The Stanley \texorpdfstring{\(g\)}{g}-kernel and weighted
Ehrhart polynomials}
\label{subsec:weighted-ehrhart-background}

Three families of polynomials will encode the preceding
decomposition: the Stanley \(g\)-polynomial records the local ranks of
the simple pure sheaves, the weighted \(h^*\)-polynomial records the
full semigroup contribution, and the weighted local
\(l^*\)-polynomial extracts its primitive multiplicities.

Following \cite[Section~2]{KatzStapledonLocal}, let
\(\mathcal P=[\hat0,\hat1]\) be an Eulerian poset of rank \(n\).
Its Stanley \(g\)-polynomial \(g(\mathcal P;T)\) is one for \(n=0\);
for \(n>0\), it is the unique polynomial of degree strictly less than
\(n/2\) satisfying
\begin{equation*}T^n g(\mathcal P;T^{-1})
 =
 \sum_{z\in\mathcal P}
 g([\hat0,z];T)
 (T-1)^{n-\operatorname{rk}(z)}.
\end{equation*}
For a cone face lattice, the interval \([c',c]\) has rank
\(\dim c-\dim c'\).  With the degree-one normalization fixed above,
the local intersection cohomology calculation reads
\begin{equation}\label{eq:simple-pure-local-g}
 \operatorname{Hilb}
 \overline{(\mathcal L_\Psi^\vartheta)_\eta}
 =
 g([\vartheta,\eta];T)
 \qquad(\vartheta\preceq\eta)
\end{equation}
\cite[Theorem~4.5]{BBFK}.  We shall also use
\begin{align}
 g(\mathcal P_1\times\mathcal P_2;T)
 &=
 g(\mathcal P_1;T)
 g(\mathcal P_2;T),
 \label{eq:stanley-g-product}\\
 \sum_{c'\preceq c''\preceq c}
 g([c',c''];T)
 (-1)^{\dim c-\dim c''}
 g([c'',c]^*;T)
 &=\delta_{c',c}.
 \label{eq:stanley-g-inversion}
\end{align}
The second identity is the standard \(g\)-kernel inversion
\cite[Theorem~3.11]{KatzStapledonLocal}.

The \(h^*\)- and \(l^*\)-polynomials are the Ehrhart-theoretic
counterparts of Stanley's theory.  The former is the numerator of a
weighted lattice-point series, whereas the latter is its local, or
primitive, part obtained by applying the dual Stanley \(g\)-kernel.
Thus \(l^*\) is more precisely analogous to a local \(h\)-polynomial
than to the \(g\)-polynomial itself; see
\cite{StapledonWeighted,KatzStapledonLocal}.

Let \(c\subset L_\R\) be a pointed rational cone of dimension \(d\),
equipped with the restriction \(\psi_c:=\psi|_c\) of a rational
height function \(\psi\); for every face \(c'\preceq c\), write
\(\psi_{c'}:=\psi|_{c'}\).  Choose \(R\in\Z_{>0}\) such that
\(\psi_c(L\cap\Span_\R c)\subseteq \frac{1}{R}\Z\).  Define
\begin{align}
 h^*_{\psi_c}(c;T)
 &:=(1-T)^d\sum_{m\in c\cap L}T^{\psi_c(m)},
 \label{eq:weighted-hstar-definition}\\
 l^*_{\psi_c}(c;T)
 &:=
 \sum_{c'\preceq c}
 h^*_{\psi_{c'}}(c';T)
 (-1)^{d-\dim c'}
 g([c',c]^*;T).
 \label{eq:weighted-lstar-definition}
\end{align}
For the zero cone, both polynomials are one.  
They belong to \(\Z[T^{1/R}]\),
see \cite[Proposition~2.6]{StapledonWeighted}.

If \(c\) is simplicial with stacky ray vectors
\(b_1,\ldots,b_d\) satisfying \(\psi_c(b_i)=1\), put
\begin{align*}
\Boxop^\circ(c)
 &:={}
 \left\{
  \sum_{i=1}^dq_ib_i\in L:0<q_i<1
 \right\}.
\end{align*}
Then the closed and open fundamental-parallelepiped formulas are
\begin{align*}
 h^*_{\psi_c}(c;T)
 &=
 \sum_{m\in\Boxop(c)}T^{\psi_c(m)},\\
 l^*_{\psi_c}(c;T)
 &=
 \sum_{m\in\Boxop^\circ(c)}T^{\psi_c(m)}.
\end{align*}

\section{Toric face double Koszul complexes and Hilbert series}
\label{sec:non-simplicial-associated-graded}

This section identifies the associated graded of the filtered
Koszul--de Rham complex with a toric face double Koszul complex.  Its
coefficient sheaf then decomposes into finite-dimensional primitive
multiplicity spaces and intrinsic support complexes.  Their respective
contributions to the final Hilbert series are the local weighted
\(h^*\)-polynomials and local intersection cohomology.

\subsection{The associated graded double Koszul complex}

Recall the convention \eqref{eq:associated-graded-convention} and set
\begin{equation*}\mathscr C_\Phi^{\Gamma,\check\Gamma}(\check g,g)
 :=
 \Gr_F
 \widehat{\mathscr C}^{\mathrm{dR},N;
 \Gamma,\check\Gamma}_\Phi(\check g,g).
\end{equation*}
The associated graded complex inherits the \((a,r)\)-bigrading of
Definition~\ref{def:toric-face-koszul-de-rham-pair}.  We give its
derived global sections the irregular Hodge bigrading fixed in
\eqref{eq:irregular-hodge-grading-convention}.
For the trivial subdivisions, we write
\(\mathscr C_\Phi(\check g,g)\).

We first describe its coefficient and universal Koszul factors.
Using the conewise polynomial sheaf of
\cref{subsec:sheaves-on-fans}, give
\(\mathscr S_\Sigma^\Gamma\) the following graded
\(\mathcal A_\Sigma\)-algebra structure
\begin{equation}\label{eq:left-bm-kappa-map}
 \kappa_{\check g,c}^\Gamma:
 \mathcal A_{\Sigma,c}
 =\operatorname{Sym}(c_\C^*)
 \longrightarrow
 \C[c]^\Gamma,
 \qquad
 \ell\longmapsto
 \partial_\ell^\Gamma\check g_c
 =
 \sum_{u\in c\cap M}
 \check\xi_u\ell(u)x^u,
 \quad
 \ell\in c_\C^*.
\end{equation}

Write the resulting module sheaf as
\(\mathcal B_{\Sigma,\check g}^\Gamma\), and define
\(\mathcal B_{\check\Sigma,g}^{\check\Gamma}\) analogously.  Put
\begin{equation*}\mathcal B_\Phi^{\Gamma,\check\Gamma}(\check g,g)
 :=
 \left.
 \bigl(
  \mathcal B_{\Sigma,\check g}^\Gamma
  \boxtimes
  \mathcal B_{\check\Sigma,g}^{\check\Gamma}
 \bigr)\right|_\Phi.
\end{equation*}
At \(\eta=(c,\check c)\), its stalk is
\begin{equation*}\mathcal B_{\Phi,\eta}^{\Gamma,\check\Gamma}(\check g,g)
 =
 \C[c]^\Gamma\otimes\C[\check c]^{\check\Gamma}.
\end{equation*}

Let
\(\mathbf m_c\in c_\C^*\otimes c_\C\) and
\(\mathbf n_{\check c}\in
\check c_\C^*\otimes\check c_\C\)
be the identity tensors.  The universal contraction--wedge
double Koszul complex is
\begin{equation}\label{eq:universal-orthogonal-clifford-complex}
 \operatorname{Cl}_{\Phi,\eta}
 :=
 \left(
  \mathcal A_{\Phi,\eta}\otimes\bigwedge^\bullet N_\C,
  \iota_{\mathbf m_c}+\mathbf n_{\check c}\wedge
 \right).
\end{equation}
Its differential squares to the universal pairing between
\(c_\C\) and \(\check c_\C\), hence to zero because
\((c,\check c)\in\Phi\).  These stalks form a complex
\(\operatorname{Cl}_\Phi\) of \(\mathcal A_\Phi\)-module sheaves.
For a homogeneous element of
\(\mathcal A_{\Phi,\eta}\otimes\bigwedge^kN_\C\),
the polynomial bidegree
\((i,j)\) plays the role of the two height degrees
\((\psi(m),\check\psi(n))\).  Hence the induced \((a,r)\)-bigrading
is
\[
 a=k+i-j,\qquad r=i+j.
\]

\begin{proposition}\label{prop:graded-total-local-form}
For \(\eta=(c,\check c)\in\Phi\), there is a canonical
identification
\begin{equation}\label{eq:graded-local-unsplit-complex}
 \mathscr C_{\Phi,\eta}^{\Gamma,\check\Gamma}(\check g,g)
 \cong
 \left(
  \C[c]^\Gamma\otimes\C[\check c]^{\check\Gamma}
  \otimes\bigwedge^\bullet N_\C,
  D_{\check g}+D_g
 \right).
\end{equation}
The identification is homogeneous for the \((a,r)\)-bigrading of
Definition~\ref{def:toric-face-koszul-de-rham-pair}; its differential
has bidegree \((0,1)\).
Moreover, the local identifications assemble to
\begin{equation}\label{eq:presented-bm-module-clifford-form}
 \mathscr C_\Phi^{\Gamma,\check\Gamma}(\check g,g)
 \cong
 \mathcal B_\Phi^{\Gamma,\check\Gamma}(\check g,g)
 \otimes_{\mathcal A_\Phi}\operatorname{Cl}_\Phi.
\end{equation}
\end{proposition}

\begin{proof}
By the bidegrees in
Definition~\ref{def:toric-face-koszul-de-rham-pair},
\(D_{\check g}\) and \(D_g\) induce the \((0,1)\)-differential on
\(\Gr_F\), whereas \(d_{\log}\), of bidegree \((1,0)\), induces
zero.  This gives \eqref{eq:graded-local-unsplit-complex}.

Under \eqref{eq:left-bm-kappa-map} and its checked analogue, the two
identity tensors in
\eqref{eq:universal-orthogonal-clifford-complex} act as
\(D_{\check g}\) and \(D_g\).  Thus the local complex is the base
change of \(\operatorname{Cl}_{\Phi,\eta}\) along
\(\mathcal A_{\Phi,\eta}\to
\mathcal B_{\Phi,\eta}^{\Gamma,\check\Gamma}(\check g,g)\).
Compatibility with face restriction proves
\eqref{eq:presented-bm-module-clifford-form}.
\end{proof}

Since \cref{thm:filtered-toric-face-direct-image} is a
quasi-isomorphism on every \(F^\lambda\), applying \(R\pi_*\) to the
distinguished triangles
\[
 F^{>\lambda}\longrightarrow F^\lambda
 \longrightarrow\Gr_F^\lambda
\]
shows that, for the subdivision induced by the product
\(\pi:\widehat\Phi\to\Phi\),
\begin{equation}\label{eq:toric-face-direct-image}
 \mathscr C_\Phi^{\Gamma,\check\Gamma}(\check g,g)
 \simeq
 R\pi_*
 \mathscr C_{\widehat\Phi}
 (\widehat{\check g},\widehat g).
\end{equation}
Taking derived global sections in
\eqref{eq:toric-face-direct-image},
with the induced Hodge bigrading
\eqref{eq:irregular-hodge-grading-convention},
gives
\begin{equation}\label{eq:toric-face-base-hilbert}
 \operatorname{Hilb}_{x,y}
 \bH^\bullet\bigl(
  \Phi,
  \mathscr C_\Phi^{\Gamma,\check\Gamma}(\check g,g)
 \bigr)
 =
 \operatorname{Hilb}_{x,y}
 \bH^\bullet\bigl(
  \widehat\Phi,
  \mathscr C_{\widehat\Phi}
  (\widehat{\check g},\widehat g)
 \bigr).
\end{equation}

\subsection{Primitive parts of the toric face modules}

We now apply the results on pure sheaves and weighted Ehrhart theory from
\cref{subsec:sheaves-on-fans,subsec:weighted-ehrhart-background}
to the two coefficient sheaves in
\eqref{eq:presented-bm-module-clifford-form}.

\begin{definition}[Facewise regularity]
\label{def:toric-face-facewise-regularity}
We say that \((\Gamma,\check g)\) is \emph{facewise regular} if,
for every cone \(c\in\Sigma\) and one (equivalently, every) basis
\(\ell_1,\ldots,\ell_{\dim c}\) of \(c_\C^*\), the sequence
\[
 \partial_{\ell_1}^\Gamma\check g_c,\ldots,
 \partial_{\ell_{\dim c}}^\Gamma\check g_c
\]
is regular in \(\C[c]^\Gamma\).  Define facewise regularity of
\((\check\Gamma,g)\) analogously.
\end{definition}

\begin{proposition}\label{prop:toric-face-coefficient-purity}
If \((\Gamma,\check g)\) and \((\check\Gamma,g)\) are facewise
regular, then
\(\mathcal B_{\Sigma,\check g}^\Gamma\) and
\(\mathcal B_{\check\Sigma,g}^{\check\Gamma}\) are pure.
\end{proposition}

\begin{proof}
We prove the assertion on \(\Sigma\).  The transverse section of
\(\Gamma|_c\) is a polyhedral ball, so \(\C[c]^\Gamma\) is Cohen--Macaulay
\cite[Corollary~4.2(ii)]{IchimRomerCanonical}.  The regular sequence
in \cref{def:toric-face-facewise-regularity} has length
\(\dim c=\dim\C[c]^\Gamma\); its quotient is therefore
finite-dimensional.  Hence \(\C[c]^\Gamma\) is finite over
\(\mathcal A_{\Sigma,c}\), and it is a maximal Cohen--Macaulay
\(\mathcal A_{\Sigma,c}\)-module.  Auslander--Buchsbaum and graded
Quillen--Suslin show that \(\C[c]^\Gamma\) is free over
\(\mathcal A_{\Sigma,c}\).

The boundary restriction
\[
 \C[c]^\Gamma\longrightarrow
 \C[\partial c]^{\Gamma|_{\partial c}}
\]
is surjective, identifies the target with the compatible boundary
sections, and kills precisely the monomials whose exponents lie in
\(\operatorname{relint}(c)\).  Thus
\(\mathcal B_{\Sigma,\check g}^\Gamma\) is pure.  The checked
assertion is identical.
\end{proof}

For \(\sigma\in\Sigma\), let
\begin{equation*}J^\Gamma(\check g_\sigma)
 :=
 \bigl(
  \partial_\ell^\Gamma\check g_\sigma:
  \ell\in\sigma_\C^*
 \bigr)
 \subseteq\C[\sigma]^\Gamma
\end{equation*}
and set
\[
 R_0^\Gamma(\check g_\sigma,\sigma)
 :=\C[\sigma]^\Gamma/J^\Gamma(\check g_\sigma),
\]
and
\begin{equation*}R_1^\Gamma(\check g_\sigma,\sigma)
 :=
 \operatorname{im}\!\left(
  \C[\operatorname{relint}(\sigma)\cap M]^\Gamma
  \longrightarrow
  R_0^\Gamma(\check g_\sigma,\sigma)
 \right).
\end{equation*}
Define
\(R_i^{\check\Gamma}(g_{\check\sigma},\check\sigma)\), for
\(i=0,1\), in the same way.

By \eqref{eq:left-bm-kappa-map}, reduction modulo the positive-degree
ideal of \(\mathcal A_{\Sigma,c}\) gives
\[
 \overline{\C[c]^\Gamma}
 =
 R_0^\Gamma(\check g_c,c).
\]
Moreover, the boundary restriction fits into the exact sequence
\[
 0\longrightarrow
 \C[\operatorname{relint}(c)\cap M]^\Gamma
 \longrightarrow \C[c]^\Gamma
 \longrightarrow
 \C[\partial c]^{\Gamma|_{\partial c}}
 \longrightarrow0.
\]
After reduction modulo the positive-degree ideal of
\(\mathcal A_{\Sigma,c}\), right exactness and
\eqref{eq:pure-multiplicity-kernel} identify the multiplicity space as
\[
 \mathcal W_c
 \bigl(\mathcal B_{\Sigma,\check g}^\Gamma\bigr)
 =
 R_1^\Gamma(\check g_c,c).
\]

\begin{proposition}\label{prop:pure-toric-face-modules}
Under the facewise regularity conditions of
\cref{def:toric-face-facewise-regularity}, there are noncanonical
decompositions
\begin{align}
 \mathcal B_{\Sigma,\check g}^\Gamma
 &\simeq
 \bigoplus_{\sigma\in\Sigma}
 R_1^\Gamma(\check g_\sigma,\sigma)
 \otimes_\C\mathcal L_\Sigma^\sigma,
 \label{eq:first-factor-pure-decomposition}\\
 \mathcal B_{\check\Sigma,g}^{\check\Gamma}
 &\simeq
 \bigoplus_{\check\sigma\in\check\Sigma}
 R_1^{\check\Gamma}(g_{\check\sigma},\check\sigma)
 \otimes_\C\mathcal L_{\check\Sigma}^{\check\sigma}.
 \label{eq:second-factor-pure-decomposition}
\end{align}
Their multiplicity spaces have Hilbert series
\begin{align}
 \operatorname{Hilb}
 R_1^\Gamma(\check g_\sigma,\sigma)
 &=
 l^*_{\psi_\sigma}(\sigma;T),
 \label{eq:r1-first-lstar}\\
 \operatorname{Hilb}
 R_1^{\check\Gamma}(g_{\check\sigma},\check\sigma)
 &=
 l^*_{\check\psi_{\check\sigma}}
 (\check\sigma;T).
 \label{eq:r1-second-lstar}
\end{align}
Thus these Hilbert series are independent of the chosen
subdivision whenever the stated regularity condition holds.  
\end{proposition}

\begin{proof}
The preceding identification of the multiplicity spaces and the
decomposition theorem for pure sheaves give
\eqref{eq:first-factor-pure-decomposition} and
\eqref{eq:second-factor-pure-decomposition}.

It remains to compute the multiplicities.  The monomial basis of
\(\C[c]^\Gamma\) is independent of the subdivision, and the Koszul
Hilbert series formula gives
\begin{equation*}\operatorname{Hilb}
 R_0^\Gamma(\check g_c,c)
 =
 (1-T)^{\dim c}
 \sum_{m\in c\cap M}T^{\psi(m)}
 =
 h^*_{\psi_c}(c;T).
\end{equation*}
Reducing \eqref{eq:first-factor-pure-decomposition} at \(c\) and
using \eqref{eq:simple-pure-local-g} yields
\begin{equation*}h^*_{\psi_c}(c;T)
 =
 \sum_{\sigma\preceq c}
 \operatorname{Hilb}
 R_1^\Gamma(\check g_\sigma,\sigma)
 g([\sigma,c];T).
\end{equation*}
Inverting this convolution by
\eqref{eq:stanley-g-inversion} gives exactly
\eqref{eq:weighted-lstar-definition}, proving
\eqref{eq:r1-first-lstar}.  
In the projective Gorenstein case, this recovers the
\(\Q\)-graded analogue of
\cite[Theorem~6.8]{BorisovMavlyutov}.
\end{proof}

\begin{remark}[A flat family of primitive parts]
\label{rem:relative-primitive-sector-bundle}
Assume in addition that the subdivision
\(\Gamma\to\Sigma\) is projective.
Fix \(c\in\Sigma\), and suppose that the facewise regularity
hypotheses hold in both \(\C[c]\) and \(\C[c]^\Gamma\).  A convex
integral support function for \(\Gamma|_c\) gives the standard flat
degeneration from \(\C[c]\) to \(\C[c]^\Gamma\).  Over the open locus
where the facewise logarithmic derivative sequences remain regular,
the corresponding primitive parts form a finite locally free
\(\Q\)-graded family whose fibres at \(t=1\) and \(t=0\) are
\(R_1(\check g_c,c)\) and
\(R_1^\Gamma(\check g_c,c)\), respectively.  Their Hilbert series are
therefore equal.  An isomorphism between the two fibres requires a
choice of trivialization of the family and hence is not canonical.
\end{remark}

\subsection{Primitive--support decomposition}

Assume the hypotheses of
\cref{prop:pure-toric-face-modules}.  For
\(\vartheta=(\sigma,\check\sigma)\in\Phi\), define the primitive
multiplicity space
\begin{equation*}V_\vartheta^{\Gamma,\check\Gamma}
 :=
 R_1^\Gamma(\check g_\sigma,\sigma)
 \otimes_\C
 R_1^{\check\Gamma}(g_{\check\sigma},\check\sigma)
\end{equation*}
and the support complex
\begin{equation*}\mathcal K_{\Phi,\vartheta}
 :=
 \mathcal L_\Phi^\vartheta
 \otimes_{\mathcal A_\Phi}\operatorname{Cl}_\Phi.
\end{equation*}
The first factor is finite-dimensional and has zero differential; it
records the primitive multiplicity attached to the chosen potentials
and subdivisions.  The second contains all face restrictions, local
intersection cohomology data, and the universal contraction--wedge
differential.  It depends only on the embedded orthogonality fan and
the lattice pairing.

\begin{proposition}\label{prop:toric-face-pure-decomposition}
There are noncanonical decompositions
\begin{align}
 \mathcal B_\Phi^{\Gamma,\check\Gamma}(\check g,g)
 &\simeq
 \bigoplus_{\vartheta\in\Phi}
 V_\vartheta^{\Gamma,\check\Gamma}
 \otimes_\C\mathcal L_\Phi^\vartheta,
 \label{eq:presented-bm-pure-decomposition}\\
 \mathscr C_\Phi^{\Gamma,\check\Gamma}(\check g,g)
 &\simeq
 \bigoplus_{\vartheta\in\Phi}
 V_\vartheta^{\Gamma,\check\Gamma}
 \otimes_\C\mathcal K_{\Phi,\vartheta}.
 \label{eq:toric-face-decomposition}
\end{align}
Consequently,
\begin{equation}\label{eq:toric-face-derived-sections-decomposition}
 R\Gamma\bigl(
  \Phi,
  \mathscr C_\Phi^{\Gamma,\check\Gamma}(\check g,g)
 \bigr)
 \simeq
 \bigoplus_{\vartheta\in\Phi}
 V_\vartheta^{\Gamma,\check\Gamma}
 \otimes_\C
 R\Gamma\bigl(\Phi,\mathcal K_{\Phi,\vartheta}\bigr).
\end{equation}
Thus derived global sections act only on the support complexes.
\end{proposition}

\begin{proof}
Take the external product of
\eqref{eq:first-factor-pure-decomposition} and
\eqref{eq:second-factor-pure-decomposition}, restrict it to
\(\Phi\), and apply
\eqref{eq:orthogonal-external-product-simple}.  The nonorthogonal
pairs vanish and the remaining terms give
\eqref{eq:presented-bm-pure-decomposition}.  Tensoring with the
termwise free \(\mathcal A_\Phi\)-complex
\(\operatorname{Cl}_\Phi\) and using
\eqref{eq:presented-bm-module-clifford-form} proves
\eqref{eq:toric-face-decomposition}.  The final assertion follows
because \(\Phi\) is finite and all multiplicity spaces are
finite-dimensional.
\end{proof}

Accordingly, the homogeneous pieces of the support complex are
\begin{equation}\label{eq:simple-support-bigrading}
 \mathcal K_{\Phi,\vartheta}^{a,r}
 :=
 \bigoplus_{i+j=r}
 (\mathcal L_\Phi^\vartheta)_{i,j}
 \otimes\bigwedge^{a-i+j}N_\C.
\end{equation}
Its differential preserves
\(a\) and raises \(r\) by one.  Define the support series
\begin{equation*}I_{\Phi,\vartheta}(x,y)
 :=
 \sum_{a,q}
 \dim_\C
 \bH^q\bigl(
  \Phi,\mathcal K_{\Phi,\vartheta}^{a,\bullet}
 \bigr)x^ay^q.
\end{equation*}
Here \(q\) is the hypercohomological degree of the support complex,
namely the sum of its internal \(r\)-degree and the cellular degree $s$.

\subsection{The intrinsic Hilbert series}

The preceding support series and local weighted \(h^*\)-polynomials
define, without choosing potentials or subdivisions,
\begin{equation}\label{eq:intrinsic-polynomial}
 \mathscr I_\Phi(x,y)
 :=
 \sum_{\vartheta=(\sigma,\check\sigma)\in\Phi}
 l^*_{\psi_\sigma}(\sigma;xy)
 l^*_{\check\psi_{\check\sigma}}
 (\check\sigma;y/x)
 I_{\Phi,\vartheta}(x,y).
\end{equation}
It depends only on the two embedded lattice fans, their height
functions, and the lattice pairing.  The sum over \(\vartheta\) is
finite.  We do not need its finiteness in the bigrading at this point;
\cref{prop:toric-face-hilbert} will show that \(\mathscr I_\Phi(x,y)\)
is a Laurent polynomial.

\begin{proposition}\label{prop:toric-face-hilbert}
Under the hypotheses of
\cref{prop:pure-toric-face-modules}, the Hilbert series for every pair
of crepant subdivisions and every pair of potentials satisfying the
stated regularity hypotheses is \(\mathscr I_\Phi\):
\begin{equation*}\operatorname{Hilb}_{x,y}
 \bH^\bullet\bigl(
  \Phi,
  \mathscr C_\Phi^{\Gamma,\check\Gamma}(\check g,g)
 \bigr)
 =
 \mathscr I_\Phi(x,y).
\end{equation*}
In particular, this Hilbert series is independent of the projective
subdivisions and of the potential coefficients within the regular
locus.
\end{proposition}

\begin{proof}
For
\(\vartheta=(\sigma,\check\sigma)\),
\eqref{eq:r1-first-lstar} and
\eqref{eq:r1-second-lstar} give
\begin{equation}\label{eq:primitive-multiplicity-hilbert}
 \operatorname{Hilb}_{U,V}
 V_\vartheta^{\Gamma,\check\Gamma}
 =
 l^*_{\psi_\sigma}(\sigma;U)
 l^*_{\check\psi_{\check\sigma}}(\check\sigma;V).
\end{equation}
Let a homogeneous element of this multiplicity space have Newton
degrees \(\nu\) and \(\check\nu\), and consider a class in
\[
 \bH^q\bigl(
  \Phi,\mathcal K_{\Phi,\vartheta}^{a,\bullet}
 \bigr).
\]
By the definitions of \(a\) and \(r\), a primitive multiplicity of
Newton bidegree \((\nu,\check\nu)\) shifts the support degrees by
\[
 a\longmapsto a+\nu-\check\nu,
 \qquad
 r\longmapsto r+\nu+\check\nu.
\]
Since \(\deg_N=a+r\) and \(q\) already includes the cellular degree, the
irregular Hodge bidegree of the corresponding tensor is
\begin{equation*}\lambda=a+\nu-\check\nu,
 \qquad
 \mu=q+\nu+\check\nu.
\end{equation*}
Equivalently,
\[
 x^\lambda y^\mu
 =
 (xy)^\nu(y/x)^{\check\nu}x^ay^q.
\]
The derived decomposition
\eqref{eq:toric-face-derived-sections-decomposition} therefore yields
\[
 \operatorname{Hilb}_{x,y}
 \bH^\bullet\bigl(
  \Phi,
  \mathscr C_\Phi^{\Gamma,\check\Gamma}(\check g,g)
 \bigr)
 =
 \sum_{\vartheta\in\Phi}
 \left.
 \operatorname{Hilb}_{U,V}
 V_\vartheta^{\Gamma,\check\Gamma}
 \right|_{U=xy,\,V=y/x}
 I_{\Phi,\vartheta}(x,y).
\]
Substitution of
\eqref{eq:primitive-multiplicity-hilbert} gives
\eqref{eq:intrinsic-polynomial}, proving the result.
\end{proof}
 \section{Stringy irregular Hodge numbers}
\label{sec:stringy-irregular-hodge-numbers}

The preceding sections compute the orbifold irregular Hodge polynomial
of a simplicial toric Landau--Ginzburg datum from a toric face complex.
We now show that the answer forgets the chosen simplicial fan and
retains only its radial support and Newton polytope at infinity.

For lattice polytopes
\[
 P\subset M_\R,
 \qquad
 Q\subset N_\R
\]
containing \(0\) and satisfying
\[
 \langle P,Q\rangle\geq0,
\]
put
\begin{equation*}\Phi_{P,Q}
 :=(\Sigma_P\oplus\Sigma_Q)_0
\end{equation*}
for the orthogonality fan of their spanning fans, and write
\begin{equation*}\mathscr I_{P,Q}(x,y)
 :=\mathscr I_{\Phi_{P,Q}}(x,y).
\end{equation*}

\begin{theorem}\label{thm:orbifold-polytope-formula}
Let \((\bSigma,f)\) be a simplicial toric Landau--Ginzburg datum and
put
\[
 P:=P_{\bSigma},
 \qquad
 Q:=P_{f,\infty}.
\]
Then
\begin{equation}\label{eq:orbifold-polytope-formula}
 \Forb(U_{\bSigma},Q;x,y)
 =\mathscr I_{P,Q}(x,y).
\end{equation}
In particular, for fixed \(Q\), the orbifold irregular Hodge
polynomial depends on \(\bSigma\) only through its radial support
\(P_{\bSigma}\subset M_\R\).
\end{theorem}

\begin{proof}
Choose a projective ray-preserving crepant simplicial subdivision
\[
 \check{\boldsymbol\Gamma}
 \longrightarrow\boldsymbol\Sigma_Q.
\]
Together with the canonical projective crepant subdivision
\(\bSigma\to\boldsymbol\Sigma_P\), it induces
\[
 \pi:\widehat\Phi
 :=(\Sigma\oplus\check\Gamma)_0
 \longrightarrow\Phi_{P,Q}.
\]
Choose the coefficients of ray-supported height-one potentials on the
two simplicial fans so that both potentials are nondegenerate at
infinity and the facewise regularity hypotheses of
\cref{def:toric-face-facewise-regularity} hold.  Such a simultaneous choice is
possible because each requirement defines a dense Zariski-open subset
of the irreducible coefficient torus.  There are only finitely many
faces and toric face rings to consider, so the intersection of these
open subsets is again dense and nonempty.  Orbifold
coefficient invariance, followed by
\cref{thm:simplicial-comparison,prop:toric-face-hilbert}, gives
\begin{equation}\label{eq:orbifold-refined-intrinsic-series}
 \Forb(U_{\bSigma},Q;x,y)
 =\mathscr I_{\widehat\Phi}(x,y).
\end{equation}

Apply \cref{prop:toric-face-hilbert} once more, now to the toric face
complex on \(\Phi_{P,Q}\) determined by
\(\Sigma\to\Sigma_P\) and
\(\check\Gamma\to\Sigma_Q\).  The direct image identity
\eqref{eq:toric-face-base-hilbert} then gives, with the chosen
potentials suppressed from the notation,
\[
 \mathscr I_{P,Q}(x,y)
 =
 \operatorname{Hilb}_{x,y}
 \bH^\bullet\!\left(
  \Phi_{P,Q},
  \mathscr C_{\Phi_{P,Q}}^{\Sigma,\check\Gamma}
 \right)
 =\mathscr I_{\widehat\Phi}(x,y).
\]
Together with \eqref{eq:orbifold-refined-intrinsic-series}, this proves
\eqref{eq:orbifold-polytope-formula}.
\end{proof}

\begin{corollary}\label{cor:intrinsic-polytope-mirror}
Let \(d:=\operatorname{rank}M=\operatorname{rank}N\).  For lattice
polytopes
\[
 P\subset M_\R,
 \qquad
 Q\subset N_\R
\]
containing \(0\) and satisfying
\(\langle P,Q\rangle\geq0\), one has
\begin{equation}\label{eq:intrinsic-polytope-mirror}
 \mathscr I_{P,Q}(x,y)
 =
 x^d\mathscr I_{Q,P}(x^{-1},y).
\end{equation}
No full-dimensionality assumption is imposed on \(P\) or \(Q\).
\end{corollary}

\begin{proof}
Choose projective crepant simplicial subdivisions
\[
 \boldsymbol\Gamma\longrightarrow\boldsymbol\Sigma_P,
 \qquad
 \check{\boldsymbol\Gamma}
 \longrightarrow\boldsymbol\Sigma_Q,
\]
and generic ray-supported height-one potentials \(g_Q\) and
\(\check g_P\) with Newton polytopes \(Q\) and \(P\), respectively.
The resulting simplicial data form a stacky Clarke mirror pair.
Harder--Lee mirror symmetry \cite[Theorem~5.16]{HarderLee} gives
\[
 \Forb(U_{\boldsymbol\Gamma},g_Q;x,y)
 =
 x^d
 \Forb(U_{\check{\boldsymbol\Gamma}},\check g_P;x^{-1},y).
\]
By \cref{thm:orbifold-polytope-formula}, the two sides are
\(\mathscr I_{P,Q}(x,y)\) and
\(x^d\mathscr I_{Q,P}(x^{-1},y)\), respectively.  This proves
\eqref{eq:intrinsic-polytope-mirror}.  
\end{proof}

\begin{definition}[Stringy irregular Hodge numbers]
\label{def:stringy-irregular-hodge-numbers}
For a toric Landau--Ginzburg datum \((\bSigma,f)\),
put
\[
 P:=P_{\bSigma},
 \qquad
 Q:=P_{f,\infty}.
\]
Its \emph{stringy irregular Hodge polynomial} is
\begin{equation*}\Fst(\bSigma,f;x,y)
 :=\mathscr I_{P,Q}(x,y),
\end{equation*}
and its \emph{stringy irregular Hodge numbers} are
\begin{equation*}f_{\mathrm{st}}^{\lambda,\mu}(\bSigma,f)
 :=[x^\lambda y^\mu]\mathscr I_{P,Q}(x,y).
\end{equation*}
Thus these invariants depend only on the pair of lattice polytopes
\((P_{\bSigma},P_{f,\infty})\); we also write them as
\(\Fst(P,Q;x,y)\) and \(f_{\mathrm{st}}^{\lambda,\mu}(P,Q)\).
\end{definition}

By \cref{thm:orbifold-polytope-formula} and orbifold coefficient
invariance, if
\(\boldsymbol\Gamma\to\bSigma\) is any projective crepant simplicial
subdivision and \(h\) is nondegenerate at infinity with
\(P_{h,\infty}=Q=P_{f,\infty}\), then
\[
 \Fst(\bSigma,f;x,y)
 =\Forb(U_{\boldsymbol\Gamma},Q;x,y)
 =\Forb(U_{\boldsymbol\Gamma},h;x,y).
\]

If \((\bSigma,f)\) and \((\cbSigma,\check f)\) form a mirror pair in
dual lattices of rank \(d\), then
\cref{cor:intrinsic-polytope-mirror} gives
\[
 \Fst(\bSigma,f;x,y)
 =
 x^d\Fst(\cbSigma,\check f;x^{-1},y).
\]

\enlargethispage{3\baselineskip}
\bibliographystyle{amsalpha}

\begin{thebibliography}{Wan26c}

\bibitem[BBFK02]{BBFK}
G.~Barthel, J.-P.~Brasselet, K.-H.~Fieseler, and L.~Kaup,
\emph{Combinatorial intersection cohomology for fans},
Tohoku Math. J. \textbf{54} (2002), 1--41.

\bibitem[Bat98]{BatyrevStringy}
V.~V.~Batyrev,
\emph{Stringy Hodge numbers of varieties with Gorenstein canonical
singularities},
in \emph{Integrable systems and algebraic geometry},
World Scientific, River Edge, NJ, 1998, 1--32.

\bibitem[BD96]{BatyrevDais}
V.~V.~Batyrev and D.~I.~Dais,
\emph{Strong McKay correspondence, string-theoretic Hodge numbers and
mirror symmetry},
Topology \textbf{35} (1996), no.~4, 901--929.

\bibitem[Bor14]{BorisovString}
L.~Borisov,
\emph{On stringy cohomology spaces},
Duke Math. J. \textbf{163} (2014), 1105--1126.

\bibitem[BCS05]{BorisovChenSmith}
L.~Borisov, L.~Chen, and G.~Smith,
\emph{The orbifold Chow ring of toric Deligne--Mumford stacks},
J. Amer. Math. Soc. \textbf{18} (2005), 193--215.

\bibitem[BM03]{BorisovMavlyutov}
L.~Borisov and A.~Mavlyutov,
\emph{String cohomology of Calabi--Yau hypersurfaces via mirror
symmetry},
Adv. Math. \textbf{180} (2003), 355--390.

\bibitem[BL03]{BresslerLunts}
P.~Bressler and V.~Lunts,
\emph{Intersection cohomology on nonrational polytopes},
Compos. Math. \textbf{135} (2003), 245--278.

\bibitem[Cla16]{ClarkeDualFans}
P.~Clarke,
\emph{Dual fans and mirror symmetry},
Adv. Math. \textbf{301} (2016), 902--933.

\bibitem[DRS10]{Triangulations}
J.~A.~De Loera, J.~Rambau, and F.~Santos,
\emph{Triangulations: Structures for algorithms and applications},
Algorithms and Computation in Mathematics, vol.~25, Springer, 2010.

\bibitem[ESY17]{ESY}
H.~Esnault, C.~Sabbah, and J.-D.~Yu,
\emph{\(E_1\)-degeneration of the irregular Hodge filtration},
J. Reine Angew. Math. \textbf{729} (2017), 171--227.

\bibitem[HL24]{HarderLee}
A.~Harder and S.~Lee,
\emph{Irregular Hodge numbers of stacky Clarke mirror pairs},
2024, arXiv:\allowbreak2408.09016v2.

\bibitem[IR09]{IchimRomerCanonical}
B.~Ichim and T.~R\"omer,
\emph{On canonical modules of toric face rings},
Nagoya Math. J. \textbf{194} (2009), 69--90.

\bibitem[KS16]{KatzStapledonLocal}
E.~Katz and A.~Stapledon,
\emph{Local \(h\)-polynomials, invariants of subdivisions, and mixed
Ehrhart theory},
Adv. Math. \textbf{286} (2016), 181--239.

\bibitem[QZ26]{QinZhang}
Y.~Qin and D.~Zhang,
\emph{Classical and irregular Hodge numbers},
2026, arXiv:2603.06040v1.

\bibitem[Sab18]{SabbahIHT}
C.~Sabbah,
\emph{Irregular Hodge theory},
M\'em. Soc. Math. Fr. (N.S.), vol.~156, 2018,
with the collaboration of J.-D.~Yu.

\bibitem[SY15]{SabbahYu}
C.~Sabbah and J.-D.~Yu,
\emph{On the irregular Hodge filtration of exponentially twisted
mixed Hodge modules},
Forum Math. Sigma \textbf{3} (2015), Paper No.~e9, 71~pp.

\bibitem[Sta08]{StapledonWeighted}
A.~Stapledon,
\emph{Weighted Ehrhart theory and orbifold cohomology},
Adv. Math. \textbf{219} (2008), 63--88.

\bibitem[Wan26a]{WangCompactification}
H.~Wang,
\emph{Compactification independence of the irregular Hodge
filtration on Deligne--Mumford stacks},
2026, arXiv:2608.06234v1.



\bibitem[Wan26b]{WangCoefficient}
\bysame,
\emph{Irregular Hodge bundles for Deligne--Mumford Landau--Ginzburg
families with fixed pole orders},
2026, manuscript.

\bibitem[Yas06]{Yasuda}
T.~Yasuda,
\emph{Motivic integration over Deligne--Mumford stacks},
Adv. Math. \textbf{207} (2006), no.~2, 707--761.

\bibitem[Yu14]{Yu}
J.-D.~Yu,
\emph{Irregular Hodge filtration on twisted de Rham cohomology},
Manuscripta Math. \textbf{144} (2014), no.~1--2, 99--133.

\end{thebibliography}
\providecommand{\bysame}{\leavevmode\hbox to3em{\hrulefill}\thinspace}
\providecommand{\MR}{\relax\ifhmode\unskip\space\fi MR }
\providecommand{\MRhref}[2]{\href{http://www.ams.org/mathscinet-getitem?mr=#1}{#2}}
\providecommand{\href}[2]{#2}

\end{document}